\documentclass[12pt]{amsart}
\usepackage{amsmath,amssymb,amsthm, amscd, mathtools}
\usepackage{pifont,mathrsfs}
\usepackage{array,lastpage}
\usepackage{enumerate,xspace,pifont,shadow, hyperref}
\usepackage{mathrsfs}
\usepackage{xfrac}
\usepackage [english]{babel}
\usepackage [autostyle, english = american]{csquotes}
\MakeOuterQuote{"}
\usepackage{xcolor}

\DeclareSymbolFont{AMSb}{U}{msb}{m}{n}
\DeclareMathSymbol{\Z}{\mathbin}{AMSb}{"5A}
\DeclareMathSymbol{\R}{\mathbin}{AMSb}{"52}
\DeclareMathSymbol{\N}{\mathbin}{AMSb}{"4E}
\DeclareMathSymbol{\Q}{\mathbin}{AMSb}{"51}

\newcommand{\NN}{\mathbb{N}}
\newcommand{\ZZ}{\mathbb{Z}}

\newcommand{\simga}{\sigma}
\newcommand{\detla}{\delta}

\newcommand{\Th}{\textup{Th}}

\newcommand{\floor}[1]{\lfloor #1 \rfloor}

\newcommand{\mc}[1]{\mathcal{#1}}

\newcommand{\ob}[1]{\overline{#1}}

\newcommand{\vareps}{\varepsilon}

\def\Ind{\setbox0=\hbox{$x$}\kern\wd0\hbox to 0pt{\hss$\mid$\hss}
\lower.9\ht0\hbox to 0pt{\hss$\smile$\hss}\kern\wd0}

\def\Notind{\setbox0=\hbox{$x$}\kern\wd0\hbox to 0pt{\mathchardef
\nn=12854\hss$\nn$\kern1.4\wd0\hss}\hbox to
0pt{\hss$\mid$\hss}\lower.9\ht0 \hbox to
0pt{\hss$\smile$\hss}\kern\wd0}

\newtheorem{thm}{Theorem}[section]
\newtheorem{lem}[thm]{Lemma}
\newtheorem{cor}[thm]{Corollary}
\newtheorem{prop}[thm]{Proposition}

\newtheorem{fact}[thm]{Fact}
\newtheorem{ass}[thm]{Assumption}
\newtheorem{quest}[thm]{Question}

\theoremstyle{definition}
\newtheorem{definition}[thm]{Definition}

\theoremstyle{remark}
\newtheorem{remark}[thm]{Remark}

\theoremstyle{remark}

\theoremstyle{remark}
\newtheorem{claim}[thm]{Claim}

\theoremstyle{remark}
\newtheorem{facta}[thm]{Fact}

\theoremstyle{remark}

\theoremstyle{remark}

\begin{document}
\bibliographystyle{plain}

\title[Discrete sets in strong OAGs]{Discrete sets definable in strong expansions of ordered Abelian groups}

\author{Alfred Dolich and John Goodrick}

\thanks{The first author's research was partially supported by PSC-CUNY Grant \#63392-00 51. The second author's research was carried out in part during a semester of paid leave supported by the Universidad de los Andes (Semestre de Trabajo Académico Independiente) and was also partially supported by the grants INV-2018-50-1424 and INV-2023-162-2840 from the Faculty of Sciences of the Universidad de los Andes.}

\address{Dept. of Math and CS \\ Kingsborough Community College (CUNY) \\ 2001 Oriental Blvd.\\ Brooklyn, NY 11235}
\address{Department of Mathematics \\ CUNY Graduate Center \\ 365 5th Ave. \\ New York, NY 10016}
\email{alfredo.dolich@kbcc.cuny.edu}

\address{Departamento de Matemáticas \\ Universidad de los Andes \\ Carrera 1 No. 18A-12 \\ Bogotá, COLOMBIA 111711}
\email{jr.goodrick427@uniandes.edu.co}

\maketitle

\begin{abstract}

We study the structure of infinite discrete sets $D$ definable in expansions of ordered Abelian groups whose theories are strong and definably complete, with particular emphasis on the set $D'$ comprised of differences between successive elements. In particular, if the burden of the structure is at most $n$, then the result of applying the operation $D \mapsto D'$ $n$ times must be a finite set (Theorem~\ref{finite_diff_set}). In the case when the structure is densely ordered and has burden $2$, we show that any definable unary discrete set must be definable in some elementary extension of the structure $\langle \R; <, +, \Z\rangle$ (Theorem~\ref{def_in_G}).

\end{abstract}

\section{Introduction}

In this article we will present new results on the structure of infinite discrete sets definable in ordered Abelian groups whose theories are strong with special emphasis on the case when the structure has finite dp-rank. Suppose that $\mc{R} = \langle R; +, < \ldots \rangle$ is a divisible ordered Abelian group (possibly with additional structure) and $D \subseteq R$ is a definable set which is discrete according to the order topology. Recall that by a result of \cite{Goodrick_dpmin}, if the theory of $\mc{R}$ is dp-minimal (that is, of dp-rank $1$), then $D$ must be finite, but there are examples of $\mc{R}$ of dp-rank $2$ in which an infinite discrete set is definable: for instance,  $\langle \R; +, <, \Z \rangle$, the expansion of  the additive group of the reals by a predicate for the set of integers (see \cite{DG} for details). On the other hand, when the theory of $\mc{R}$ is strong (in particular, when it has finite dp-rank), then by general results from \cite{DG}, $D$ cannot have accumulation points, and furthermore there must exist a $\delta \in R$ such that for any $a \in D$, we have $|(a-\delta, a + \delta) \cap D| > 1$ (that is, the points in $D$ cannot be ``too spread out'').

The present article presents stronger results on definable discrete sets $D$ when $\mc{R}$ has finite burden. The first set of results (in Section 2) concerns the ``difference set'' $D'$ of $D$. Assuming that $\mc{R}$ is definably complete (see Definition~\ref{def_complete} just below), for every non-maximal $a \in D$ there is a next largest element $S_D(a)$ in $D$, and we define $D' = \{S_D(a) - a \, : \, a \in D \setminus \max(D)\}$. By Fact~\ref{discrete_image} below, $D'$ is also a discrete definable set. We will show that when $\mc{R}$ has burden at most $2$, then $D'$ is finite. More generally, we can iterate the operation of taking difference sets, letting $D^{(0)} = D$ and $D^{(k+1)} = (D^{(k)})'$, and we show the following:

\hypertarget{fds}{\begin{thm}}
\label{finite_diff_set}
Suppose that $\mc{R}$ is a definably complete ordered Abelian group, $D \subseteq R$ is definable and discrete, and $D^{(n)}$ is infinite.  Then the burden of $\mc{R}$ is  at least $n+1$. If $\mc{R}$ is densely ordered, then the burden is greater than $n+1$.
\end{thm}

\begin{quest}
Suppose that $\mc{R}$ is a definably complete ordered Abelian group of finite dp-rank. Are there only finitely many Archimedean classes represented in $D'$?
\end{quest}

We conjecture that the answer to the Question above is ``yes,'' but we were not able to show this.


The second part of the article (Section 3) focuses on the fine structure of a discrete set $D \subseteq R$ definable in $\mc{R}$ under the stronger hypothesis that $\mc{R}$ has burden $2$ and is definably complete. We will show the following:

\begin{thm}
\label{def_in_G}
Suppose that $\mc{R}=\langle R; +,<, \dots \rangle$ is a densely-ordered Abelian group of burden $2$ which is definably complete and in which there is an infinite discrete subset of $R$ definable.   There is $G \subseteq R$ with $\langle R; +, <, G\rangle \equiv \langle \R; +, <, \ZZ\rangle$ so that any discrete
$D \subseteq R$ definable in $\mc{R}$  is definable in $\langle R; +, <, G\rangle$.  Furthermore if $\mc{R}$ is of dp-rank $2$  then there is such a $G$ so that any definable $X \subseteq R$ is definable in $\langle R; +, <, G\rangle$.
\end{thm}

Notice that this is very similar to a result from \cite{DG} which has an analogous conclusion in the  case when the divisible ordered Abelian group is Archimedean but under the more general hypothesis that the theory is strong.

 
 In addition to proving Theorem~\ref{def_in_G}, we obtain a fairly precise description of discrete sets definable in $\mc{R}$, showing that they are unions of finitely many points and ``pseudo-arithmetic'' sets (that is, sets $E$ such that $|E'| = 1$). Note that this has a natural analogue for discrete sets definable in strong Archimedean ordered Abelian groups, which were shown to be finite unions of points and arithmetic progressions in \cite{DG}, and in fact part of the proof strategy from that previous paper is adapted here to the non-Archimedean context.

In the remainder of this introduction we will briefly recall the definitions of the central concepts of the article (especially dp- and inp-rank, burden, and definable completeness). For a more thorough background on the various notions involved and their importance the reader is referred to the book of Simon \cite{Guide_NIP}. This article continues work begun in \cite{DG}, \cite{Goodrick_dpmin}, \cite{S}, and especially the recent note \cite{topological_properties_burden_2} in which the tame topological properties of sets definable in burden-$2$ ordered Abelian groups are studied.

\subsection{Notations and convention}

Most of our notation and terminology is standard for model theory, but for convenience we recall here some concepts which are not universally well-known, such as burden and dp-rank.

Structures will be named by calligraphic capital letters such as $\mathcal{R}$ and $\mathcal{U}$, with block letters (like $R$ and $U$) used to denote their underlying universes. For example, $\mathcal{R} = \langle R; +, <, \ldots \rangle$ denotes a structure $\mathcal{R}$ with universe $R$, a function symbol $+$, a binary relation symbol $<$, and possibly other basic predicates and functions.

We will work with \emph{ordered Abelian groups}, or Abelian groups $\langle R; + \rangle$ endowed with a strict linear ordering $<$ which is invariant under the group operation ($x < y$ implies that $x + z < y + z$). We abbreviate ``ordered Abelian group'' as OAG. Note that an OAG can be discretely ordered (having a least positive element, such as $\langle \Z; <, +\rangle$) or densely ordered (such as $\langle \R; <, + \rangle$). Recall that an Abelian group $\langle R; + \rangle$ is \emph{divisible} if for every $g \in R$ and every positive integer $n$, there is an $h \in R$ such that $nh = g$. Every divisible OAG is densely ordered, but there are densely-ordered OAGs which are not divisible.

The ordering on an OAG $\langle R; <, + \rangle$ naturally defines an order topology on $R$ whose basic open sets are open intervals, and by default any topological concepts (being open, closed, discrete, \emph{et cetera}) refers to this order topology, or to the corresponding product topology on $R^n$.

Throughout, we will use ``definable'' to mean ``definable by some first-order formula, \emph{possibly using extra parameters}.''

Whenever we talk about ``intervals'' in a densely-ordered structure, by default we mean nonempty open intervals unless specified otherwise.

\begin{definition}
\label{def_complete}
If $\mathcal{R} = \langle R; <, +, \ldots, \rangle$ is an expansion of an OAG, then $\mathcal{R}$ is \emph{definably complete} if  every nonempty definable subset $X \subseteq R$ which has an upper bound in $R$ has a least upper bound in $R$.
\end{definition}

The following observation (\cite[Proposition 2.2]{ivp}) will sometimes be useful in what follows:

\begin{fact}
\label{DC_divisible}
If $\mathcal{R}$ is an expansion of a  densely ordered OAG which is definably complete, then $\mathcal{R}$ is divisible.
\end{fact}

Finally, we recall the definitions of burden and dp-rank. These notions are originally due to Shelah \cite{strong_dep}, but the form of the definition which we give is due to Adler \cite{adler_strong_dep}. In the definitions below, we work in some fixed complete first-order theory $T$. Our theories are always $1$-sorted, since in $T^{eq}$ there will always be sorts of dp-rank greater than $1$ (assuming there are infinite models).

\begin{definition}
\label{inp}
An \emph{inp-pattern of depth $\kappa$} is a sequence $\{ \varphi_i(\overline{x}; \overline{y}) \, : \, i < \kappa \}$ of formulas, a sequence $\{k_i \, : \, i < \kappa\}$ of positive integers, and a sequence $\{\overline{a}_{i,j} \, : \, i < \kappa, j < \omega \}$ of tuples from some model $\mathcal{M} \models T$ such that:

$\bullet$ For each $i < \kappa$, the ``$i$-th row'' 
\begin{equation}
\{\varphi_i(\overline{x}; \overline{a}_{i, j}) \, : \, j < \omega \}
\end{equation}
 is $k_i$-inconsistent; and

$\bullet$ For each function $\eta \, : \, \kappa \rightarrow \omega$, the set of formulas

\begin{equation}
\{ \varphi_i(\overline{x}; \overline{a}_{i, \eta(i)}) \, : \, i < \kappa\}
\end{equation}

is consistent.

If $p(\overline{x})$ is a partial type, an inp-pattern as above is \emph{in $p(\overline{x})$} if every partial type as in (2) is consistent with $p(\overline{x})$.

The partial type $p(\overline{x})$ has \emph{burden less than $\kappa$} if there is \textbf{no} inp-pattern of depth $\kappa$ in $p(\overline{x})$. If the least $\kappa$ such that the burden of $p(\overline{x})$ is less than $\kappa$ is a successor cardinal, say $\kappa = \lambda^+$, then we say that the burden of $p(\overline{x})$ is $\lambda$.

If the burden of the partial type $x=x$ (in a single free variable $x$) exists in the theory $T$ and is equal to $\kappa$ then we say that \emph{the burden of $T$ is $\kappa$}. The theory $T$ is \emph{inp-minimal} if its burden is $1$.

The theory $T$ is \emph{strong} if every inp-pattern in $T$ has finite depth.

\end{definition}

Following Adler \cite{adler_strong_dep}, the notion of dp-rank is usually defined in terms of arrays of formulas known as \emph{ict-patterns} (or \emph{randomness patterns}, in \cite{onsh_usv}) which are similar to the inp-patterns above. However, for the arguments in the present paper, we will not need to work with ict-patterns, preferring to use inp-patterns combined with the following fact:

\begin{fact}\footnote{This fact was originally proved in unpublished notes by Adler \cite{adler_strong_dep}.}
\begin{enumerate}
\item (See \cite{Guide_NIP}, Observation 4.13) The dp-rank of a theory $T$ is less than $|T|^+$ if and only if $T$ is NIP. 
\item (See \cite{onsh_usv}, Lemma 2.11 (iv)) In case $T$ is NIP, the dp-rank of any partial type in $T$ is equal to its burden. In particular, $T$ is dp-minimal (that is, of dp-rank $1$) if, and only if, $T$ is both NIP and inp-minimal.
\end{enumerate}
\end{fact}

\section{Difference Sets for Discrete sets in strong OAGs}

In this section we provide a detailed analysis of discrete sets definable in a definably complete expansion of an ordered Abelian group whose theory is strong.  The hypothesis of definable completeness is useful for defining a ``successor function'' on definable discrete sets (see the definition of $S_D$ below). The principal result of this section is \hyperlink{fds}{Theorem~\ref{finite_diff_set}} showing that if $\mc{R}$ has burden at most $n+1$ and $D$ is a discrete set definable in $\mc{R}$, then the $n$-fold difference set $D^{(n)}$ (as defined below) must be finite. 

Note that for the case when $\mc{R}$ has burden $2$, the conclusion of Theorem~\ref{finite_diff_set} is that if $D \subseteq R$ is definable and discrete, then $D'$ is finite. In a previous paper, we derived the same conclusion under the assumption that the universe of $\mc{R}$ is Archimedean and the complete theory of $\mc{R}$ is strong (see Corollary 2.29 of \cite{DG}). In the current paper, we generally work in $\omega$-saturated models, so the results of \cite{DG} for Archimedean OAGs cannot be applied. The motivation of many of the lemmas here and in the following section was to generalize our earlier results to the non-Archimedean context, compensating with stronger assumptions on the theory (finite burden or burden $2$).


\begin{ass}
Throughout the remainder of this section \[\mc{R} = \langle R; +, <, \ldots \rangle\] denotes an expansion of an OAG which is definably complete, strong, and sufficiently saturated (generally $\omega$-saturated will be enough), and $D \subseteq R$ is a definable set which is discrete. $T$ denotes the complete first-order theory of $\mc{R}$.
\end{ass}

\begin{remark}
In this section, we do not generally assume that $\mc{R}$ is densely ordered, though  in a couple of places (notably Theorem~\ref{finite_diff_set}) we can achieve slightly stronger conclusions when $\mc{R}$ is not discrete. In case $\mc{R}$ is a discretely ordered group, a ``discrete subset of $R$'' will simply mean \emph{any} subset of $R$, and thus many of our results apply generally to \textbf{all} unary definable sets of a discretely ordered Abelian group.
\end{remark}



We first recall a couple of useful facts about families of discrete definable sets from \cite{DG}. 


\begin{fact}\label{closedd}
(\cite{DG}, Corollary 2.13) No definable discrete subset $D$ of $R$ can have an accumulation point. In particular, any definable discrete set $D \subseteq R$ is closed.  As a consequence, the union of finitely many definable discrete sets is also closed and discrete.
\end{fact}

\begin{fact}
\label{discrete_image}
(\cite{DG}, Corollary 2.17) If $D \subseteq R$ is discrete and definable and $f: R^n \rightarrow R$ is any definable function, then the image set $f[D^n]$ is also discrete.
\end{fact}

Next we set some basic definitions which will be used throughout the remainder of the paper: the successor function $S_D$, $\ZZ$-chains in a discrete set, and Archimedean equivalence.

\begin{definition}  Let $D$ be a discrete set definable in $\mc{R}$.
\begin{enumerate}
\item If $a \in D$ is not maximal, set $S_D(a)=\min\{b \in D : a <b\}$.\footnote{Note that this minimum element always exists by definable completeness.}  For $n \in \NN^{>0}$ we let $S_D^{n}$ be the $n$-th iterate of $S_D$ (when defined).  Similarly if $a$ is not minimal in $D$ we let $S_D^{-1}(a)=\max\{b \in D : b<d\}$ and for $n \in \NN^{>0}$ we let $S_D^{-n}$ be the $n$-th iterate (when defined).  Finally let $S_D^{0}(a)=a$ for any $a \in D$.
\item If $a$ is not maximal in $D$ we let $\gamma_D(a)=S_D(a)-a$, which we call the {\em difference} at $a$.
\item $D'=\{\gamma_D(a) : a \in D  \text{ is not maximal}\}$, the {\em difference set} of $D$.
\item Note that $D'$ is the image of $D$ under the definable map $a \mapsto \gamma_D(a)$, so by Fact~\ref{discrete_image}, the set $D'$ is also discrete. Hence we may define the \emph{$n$-th difference set} $D^{(n)}$ by recursion on $n \in \N$: $D^{(0)} = D$, and $D^{(n+1)} = (D^{(n)})'$.
\end{enumerate}

\end{definition}

\begin{definition}\label{zchain}  Let $D$ be a discrete definable set in $\mc{R}$ and let $a \in D$.  The {\em $\ZZ$-chain} of $a$ is the set:
\[\mathcal{Z}(a)=\{S_D^{n}(a) : n \in \ZZ \text{ and } S_D^{n}(a) \text{ exists}\}.\]  Any subset of $D$ of the form $\mc{Z}(a)$ for some $a \in D$ will be referred to as a {\em $\ZZ$-chain of $D$}.  We also write:
 \[\mc{Z}_{\geq}(a)=\{S_D^n(a): n \geq 0 \text{ and } S_D^n(a) \text{ exists}\}\]
and \[\mc{Z}_{\leq}(a)=\{S_D^n(a): n \leq 0 \text{ and } S_D^n(a) \text{ exists}\}.\]
Any subset of $D$ of the form $\mc{Z}_{\geq}(a)$ is called an \emph{$\omega$-chain of $D$}, and any subset of the form $\mc{Z}_{\leq}(a)$ is called an \emph{$\omega^*$-chain of $D$}.

\end{definition}

\begin{definition}  For $a \in \mc{R}^{>0}$ the {\em Archimedean class} of $a$ is the set:  
\[ [a]=\{b \in R :  b < na \text{ and } a<nb \text{ for some } n \in \NN\}.\]  Note that the collection of Archimedean classes partitions $\mc{R}^{>0}$.  We write $a \ll b$ if $[a]<[b]$ and for sets $X,Y \subseteq \mc{R}^{>0}$ we write $X \ll Y$ if for all $x \in X$ and $y \in Y$ we have $x \ll y$.
Two elements $a$ and $b$ of $R$ are called \emph{Archimedean equivalent} if they belong to the same Archimedean class.
\end{definition}

We record the following fact about definable families of discrete sets which we believe is interesting, even though it will not be necessary in proving our main results.

\begin{prop}\label{union}  Suppose that $D \subseteq R$ is discrete and definable and that $D_a$ is a uniformly definable family of discrete sets for $a \in D$.  Then $\bigcup_{a \in D}D_a$ is discrete.
\end{prop}

\begin{proof}  We may assume that $\mathcal{R}$ is densely ordered as otherwise the statement is trivial. 

For each $a \in D$ and $b\in D_a$ there is $\vareps>0$ so that $(b-\vareps, b+\vareps) \cap D_a =\{b\}$.  For $b \in D_a$ let $f_a(b)=\sup\{\vareps>0 : 
(b-\vareps, b+\vareps) \cap D_a = \{b\}\}$ (which exists in $R$ by definable completeness).  This yields a definable function $f_a: D_a \to R$. Observe that $f_a[D_a]$ is discrete (by Fact~\ref{discrete_image}), hence $f_a[D_a]$ is closed (Fact~\ref{closedd}), and $f_a[D_a]>0$.  Hence for every $a \in D$ there is $\vareps>0$ so that $\vareps < f_a[D_a]$.  Define $g: D \to R$ by the rule $$g(a)=\sup\{\vareps : \vareps<f_a[D_a]\}$$ (so by definable completeness, $g(a) = \min ( f_a[D_a])$).  This is a definable function, so $g[D]$ is discrete (by Fact~\ref{discrete_image}) and $g[D]>0$.  Pick $\vareps^*>0$ so that $2\vareps^*<g[D]$.  Thus for any $a \in D$ and any $b \in D_a$ we have that $(b-2\vareps^*, b+2\vareps^*) \cap D_a=\{b\}$.

Now suppose  that $b$ is an accumulation point of $\bigcup_{a\in D}D_a$.  By our choice of $\vareps^*$, for any fixed $a \in D$, there cannot be more than one point of $D_a$ in an open interval of length $2 \vareps^*$, so in particular there is at most one point of $D_a$ in $(b - \vareps^*, b + \vareps^*)$. Thus we can define a function 
$h: D \to R$ such that $h(a)$ is the unique element of $D_a \cap (b-\vareps^*,b+\vareps^*)$, if this exists, and otherwise $h(a) = b$.  Since $h$ is a definable function, its image $h[D]$ is discrete by Fact~\ref{discrete_image} again.  But $b$ must be an accumulation point of $h[D]$, contradicting Fact~\ref{closedd}.  Therefore $\bigcup_{a\in D}D_a$ is discrete.
\end{proof} 

Our next goal is to prove Theorem~\ref{descending_prime} below, which says that the existence of definable discrete sets $D_0, \ldots, D_n$ with certain properties implies the existence of an inp-pattern of depth $n+1$; our main result, Theorem~\ref{finite_diff_set}, will follow rather quickly from this. Before proving Theorem~\ref{descending_prime}, we will need the next definition below and a pair of technical lemmas (Lemmas~\ref{sums} and \ref{prime_reduction}) which will be useful in the eventual construction of an inp-pattern.

\begin{definition}
If $D_0, D_1, \ldots, D_n$ are subsets of $R$, then $$D_0 + D_1 + \ldots + D_n = \{c_0 + c_1 + \ldots + c_n \, : \, \forall i \leq n \, \left[ c_i \in D_i \right] \}.$$
\end{definition}

\begin{lem} \label{discrete_finite_sums}
If $D_0, \ldots, D_n$ are definable discrete sets, then $D_0 + \ldots + D_n$ is discrete.
\end{lem}

\begin{proof}
This Lemma could be quickly deduced from Proposition~\ref{union} above, but it can also be proved directly, as follows. Let $D = D_0 \cup \ldots \cup D_n$, which is discrete as it is a finite union of discrete sets. Fix some arbitrary point $b \in D_0 + \ldots + D_n$ and define $f \, : \, D^{n+1} \rightarrow R$ by the rule

\begin{equation*}
    f(c_0, \ldots, c_n) =
    \begin{cases*}
      c_0 + \ldots + c_n, & if $c_i \in D_i$ for all $i$; \\
      b        & otherwise.
    \end{cases*}
  \end{equation*}
  
  Then $D_0 + \ldots + D_n$ is the image of $D$ under $f$, and hence by Fact~\ref{discrete_image} it is discrete.
\end{proof}

\begin{lem} \label{sums} Suppose that $\widetilde{D}_0, \ldots, \widetilde{D}_n$ are definable infinite discrete sets such that:

\begin{enumerate}
\item for every $i \in \{0, \ldots, n\}$,  the set $\widetilde{D}_i$ is is bounded above, and $0 < \widetilde{D}_i$; and
\item  for every $i$ such that $i \in \{0, \ldots n-1\}$, $$\max(\widetilde{D}_{i+1}) + \min(\widetilde{D}'_{i+1}) < \min(\widetilde{D}'_i).$$
\end{enumerate}

Then the ordering on $\widetilde{D}_0 + \ldots + \widetilde{D}_n$ as a subset of $R$ is isomorphic to the lexicographic ordering on $\widetilde{D}_0 \times \ldots \times \widetilde{D}_n$ under the natural map sending the tuple $(c_0, \ldots, c_n) \in \widetilde{D}_0 \times \ldots \times\widetilde{D}_n$ to the sum $c_0 + \ldots + c_n.$

In particular, if $c_n$ is a non-maximal element of $\widetilde{D}_n$, then

$$S_{\widetilde{D}_0 + \ldots + \widetilde{D}_n}(c_0 + \ldots + c_{n-1} + c_n) = c_0 + \ldots + c_{n-1} + S_{\widetilde{D}_n}(c_n).$$

\end{lem}

\begin{proof}

First, we establish the following:

\begin{claim} \label{gamma_sum}
If $0 \leq i  < n$, $c_i \in \widetilde{D}_i$, and $c_i \neq \max(\widetilde{D}_i)$, then for any tuple $(c_{i+1}, c_{i+2}, \ldots, c_n) \in \widetilde{D}_{i+1} \times \widetilde{D}_{i+2} \times \ldots \times \widetilde{D}_n$, $$\gamma_{\widetilde{D}_i}(c_i) > c_{i+1} + \ldots + c_n.$$
\end{claim}

\emph{Proof of Claim \ref{gamma_sum}:} Fix some $i \in \{0, \ldots, n-1\}$ and fix some $(c_i, \ldots, c_n) \in \widetilde{D}_i \times \ldots \times \widetilde{D}_n$ such that $c_i \neq \max(\widetilde{D}_i)$. Then

\begin{align*}
\gamma_{\widetilde{D}_i}(c_i)  \geq \min (\widetilde{D}_i') &>  \max(\widetilde{D}_{i+1}) + \min(\widetilde{D}'_{i+1}), &\, \, \,  \textup{by hypothesis (2)} \\
 & \geq   c_{i+1} + \min(\widetilde{D}'_{i+1}) & \\
 & > c_{i+1} + \max(\widetilde{D}_{i+2}) + \min(\widetilde{D}'_{i+2}), & \, \, \, \textup{by hypothesis (2)}\\
 & \geq c_{i+1} + c_{i+2} + \min(\widetilde{D}'_{i+2}) & \\
 & \vdots & \\
 & \geq c_{i+1} + \ldots + c_{n-1} + \min(\widetilde{D}'_{n-1}) & \\
 & > c_{i+1} + \ldots + c_{n-1} + \max(\widetilde{D}_n) + \min(\widetilde{D}'_n) & \, \, \, \textup{by hypothesis (2)}\\
 &> c_{i+1} +  \ldots + c_{n-1} + \max(\widetilde{D}_n)\\
 &\geq c_{i+1} + \ldots + c_{n-1} + c_n,
\end{align*}

proving the Claim.

\medskip

To prove Lemma~\ref{sums}, consider any two tuples $(c_0, \ldots, c_n)$ and $(d_0, \ldots, d_n)$ from $\widetilde{D}_0 \times \ldots \times \widetilde{D}_n$, and we must show that $$c_0 + \ldots + c_n < d_0 + \ldots + d_n$$ if and only if the tuple $(c_0, \ldots, c_n)$ comes before $(d_0, \ldots, d_n)$ in the lexicographic order. Without loss of generality, the tuples are distinct and $(c_0, \ldots, c_n)$ comes before $(d_0, \ldots, d_n)$ in the lexicographic order. Let $i \in \{0, \ldots, n\}$ be minimal such that $c_i \neq d_i$, so that $c_i < d_i$. We can further assume that $i < n$, since if $i = n$ the conclusion we want is trivial. Note that

\begin{equation}
\label{D_i_successor}
d_i \geq c_i + \gamma_{\widetilde{D}_i}(c_i).
\end{equation}

Hence

\begin{align*}
\sum_{j=0}^n (d_j - c_j) &= d_i - c_i + \sum_{j=i+1}^n (d_j - c_j),& \textup{ by minimality of } i\\
&\geq \gamma_{\widetilde{D}_i}(c_i) + \sum_{j=i+1}^n (d_j - c_j),& \textup{ by (\ref{D_i_successor})}\\
&> \sum_{j=i+1}^n d_j,& \textup{ by  Claim  \ref{gamma_sum}}\\
&> 0,& \textup{ by Hypothesis (1).}
\end{align*}

Therefore $\sum_{j=0}^n c_j < \sum_{j=0}^n d_j$, as we wanted.

\end{proof}

\begin{lem}
\label{prime_reduction} If there are definable infinite discrete sets $D_0, \ldots, D_n$ such that

(a) $0 < D_0$, and

(b) for every $i$ such that $0 \leq i < n$, $$0 < D_{i+1} < D'_i,$$

then there are definable infinite discrete sets $\widetilde{D}_0, \ldots, \widetilde{D}_n$ satisfying hypotheses (1) and (2) of Lemma~\ref{sums}.
\end{lem}

\begin{proof}
Note that hypothesis (b) implies that $D_1, \ldots, D_n$ are all bounded above. In case $D_0$ is not bounded above, we may (using $\omega$-saturation) replace it with an infinite subset which is bounded above, which will not affect hypothesis (b); so without loss of generality, every $D_i$ is bounded above.

Now for each $i \in \{1, \ldots, n\}$, let $d_i$ be the largest element of $D_i$, let $c_i$ be the second largest element of $D_i$, and let $b_i$ be the third largest element of $D_i$, and we define

 \[ \widetilde{D}_i :=\begin{cases} (D_i \setminus \{c_i, d_i\}) \cup \{ d_i - c_i + b_i \}, & \text{ if } b_i < 2 c_i - d_i; \\ D_i \setminus \{d_i\}, & \text{ if } b_i \geq 2 c_i - d_i .\end{cases}\] 

Finally, we let $$\widetilde{D}_0 := D_0.$$

Note that when $b_i < 2 c_i - d_i$, we have $$d_i - c_i + b_i = c_i - (2 c_i - d_i - b_i) < c_i,$$ while if $b_i \geq 2 c_i - d_i$, then $\max(\widetilde{D}_i) = c_i$; thus in either case, for any $i \in \{1, \ldots, n\}$,

\begin{equation}
\max(\widetilde{D}_i) \leq c_i.
\end{equation}

Also, since $b_i \in D_i > 0$ and $c_i < d_i$, it follows that $d_i - c_i + b_i >0$, and so in either case of the definition of $\widetilde{D}_i$, we have

\begin{equation}
\widetilde{D}_i > 0.
\end{equation}

Next we will verify that for any $i \in \{0, \ldots, n\}$,

\begin{equation}
\widetilde{D}_i' \subseteq D'_i.
\end{equation}

The relation (6) is trivial when $i = 0$ and immediate from the definition when $b_i \geq 2 c_i - d_i$, so say $i > 0$ and $b_i < 2 c_i - d_i$. Since $$d_i - c_i + b_i = b_i + (d_i-c_i) > b_i,$$ it follows that $d_i - c_i + b_i$ is the largest element of $\widetilde{D}_i$ and $b_i$ is the second largest element of $\widetilde{D}_i$. Hence

\begin{align*}
\widetilde{D}_i' &\subseteq D_i'  \cup \{d_i - c_i + b_i - b_i \}\\
&= D_i' \cup \{d_i - c_i \}\\
&= D'_i,
\end{align*}

as desired.

Now from (6) we immediately deduce that for any $i \in \{0, \ldots, n\}$,

\begin{equation}
\min(D'_i) \leq \min(\widetilde{D}'_i).
\end{equation}

Finally, we will show that for any $i \in \{0, \ldots, n-1\}$,

\begin{equation}
\min(\widetilde{D}_{i+1}') \leq d_{i+1} - c_{i+1}.
\end{equation}

To prove (8), we divide into two cases depending on the definition of $\widetilde{D}_{i+1}$. On the one hand, if $b_{i+1} < 2 c_{i+1} - d_{i+1}$, then, as noted above, $b_{i+1}$ and $d_{i+1} - c_{i+1} + b_{i+1}$  are the two largest elements of $\widetilde{D}_{i+1}$, so that

\begin{align*}
\gamma_{\widetilde{D}_{i+1}}(b_{i+1}) &= S_{\widetilde{D}_{i+1}}(b_{i+1}) - b_{i+1}\\
&= d_{i+1} - c_{i+1} + b_{i+1} - b_{i+1}\\
&= d_{i+1} - c_{i+1},
\end{align*}

so that

$$\min(\widetilde{D}_{i+1}') \leq \gamma_{\widetilde{D}_{i+1}}(b_{i+1}) = d_{i+1} - c_{i+1},$$

as desired.

On the other hand, suppose that $i \in \{0, \ldots, n-1\}$ and $b_{i+1} \geq 2 c_{i+1} - d_{i+1}$. Then $$\widetilde{D}_{i+1} = D_{i+1} \setminus \{d_{i+1}\},$$ so in particular $c_{i+1} - b_{i+1} \in \widetilde{D}_{i+1}'$, and

\begin{align*}
\min(\widetilde{D}_{i+1}') &\leq c_{i+1} - b_{i+1}\\
 &\leq c_{i+1} - (2 c_{i+1} - d_{i+1})\\
 &= d_{i+1} - c_{i+1},
 \end{align*}
 
 and again we have the desired conclusion (8).

 We will now check that hypotheses (1) and (2) of Lemma~\ref{sums} hold for $\widetilde{D}_i$. First, hypothesis (1) of Lemma~\ref{sums} simply says that $\widetilde{D}_i$ is bounded above and $0 < \widetilde{D}_i$ for each $i \in \{0, \ldots, n\}$, which is (5). As for hypothesis (2) of that Lemma, suppose $i \in \{0, \ldots, n-1\}$. 
Hence


\begin{align*}
\max(\widetilde{D}_{i+1}) + \min(\widetilde{D}_{i+1}') &\leq c_{i+1} + \min(\widetilde{D}_{i+1}'), \, \, \, \, \textup{by (4)}\\
&\leq c_{i+1} + (d_{i+1} - c_{i+1}), \, \, \, \, \textup{by (8)}\\
&= d_{i+1}\\
&= \max(D_{i+1})\\
& < \min(D'_i), \, \, \, \, \textup{by hypothesis (b) of Lemma~\ref{prime_reduction}}\\
&\leq \min(\widetilde{D}'_i), \, \, \, \, \textup{by (7)}.
\end{align*}

This shows that both hypotheses of Lemma~\ref{sums} hold with the $\widetilde{D}_i$ in place of $D_i$, as desired.

\end{proof}

\begin{thm} \label{descending_prime} If there are definable infinite discrete sets $D_0, \ldots, D_n$ such that 

(a) $0 < D_0$, and

(b) for every $i$ such that $0 \leq i < n$, $$0 < D_{i+1} < D'_i,$$

then the burden of $\mathcal{R}$ is at least $n+1$. If we additionally assume that  $\mc{R}$ is densely ordered, then its burden is at least $n+2$.
\end{thm}

\begin{proof}
First, by Lemma~\ref{prime_reduction}, there are infinite definable discrete sets $\widetilde{D}_0, \ldots, \widetilde{D}_n$ satisfying the conclusions of Lemma~\ref{sums}. We will use the sets $\widetilde{D}_0, \ldots, \widetilde{D}_n$ to construct an inp-pattern of depth $n+1$, and at the last step explain how the construction can be extended to add an extra row to the pattern in case $\mc{R}$ is densely ordered. 

First we will describe the formulas $\varphi_i(x; \overline{y}_i)$ of Row $i$ for each $i \in \{0, \ldots, n+1\}$, and afterwards we will explain how to select parameters $\overline{a}_{i,j}$ so that the $\varphi_i(x;\overline{a}_{i,j})$ form an inp-pattern.

\medskip

\underline{Row $0$:} The formula $\varphi_0(x; \overline{y})$ for the first row (Row $0$) has variables $x$ and $\overline{y} = (y_0, y_1)$ and is

$$\varphi_0(x; \overline{y}) = y_0 < x < y_1,$$

asserting that $x$ lies in a certain open interval.

\medskip

\underline{Rows $1$ through $n+1$:} If $1 \leq i \leq n$, the formula $\varphi_i(x; \overline{y})$ has variables $x$ and $\overline{y} = (y_0, y_1)$ in addition to the parameters used to define the sets $\widetilde{D}_0, \ldots, \widetilde{D}_n$ (which we hold constant and never state explicitly). The formula for Row $i$ is

$$\varphi_i(x; \overline{y}) = y_0 < \min \left\{ x - z \, : \, z \in \widetilde{D}_0 + \ldots + \widetilde{D}_{i-1} \textup{ and } z < x \right\} < y_1,$$

observing that the minimum above exists for any value of $x$ (by definable completeness) and is definable. Also note that this formula makes sense for $i$ up to and including $n+1$, although we will only use $\varphi_{n+1}(x;\overline{y})$ in our inp-pattern in the case when $\mc{R}$ is densely ordered. 

\medskip

\underline{Selecting the parameters:} Finally, we will explain how to select parameters $\overline{a}_{i,j}$ for the formulas $\varphi_i(x;\overline{y})$ so that $$\{\varphi_i(x; \overline{a}_{i,j}) \, : \, i < n+1, j < \omega\}$$ forms an inp-pattern (and in case $\mc{R}$ is densely ordered, we include $i = n+1$). We will assume without loss of generality that $\mc{R}$ is $\omega$-saturated and we will select all $\overline{a}_{i,j}$ from $R$.

\medskip

First we will pick the parameters for Rows $0$ through $n$, and at the end we will explain how to pick parameters for Row $n+1$ in case $\mathcal{R}$ is densely ordered.

Recall that for Rows $0$ through $n$ to form an inp-pattern, it suffices to pick the parameters $\overline{a}_{i,j}$ such that:

\begin{enumerate}
\item for every $i \in \{0, \ldots, n\}$, the formulas $\{\varphi_i(x; \overline{a}_{i,j}) \, : \, j < \omega\}$ of Row $i$ are $2$-inconsistent; and
\item for every function $\eta \, : \, \{0, \ldots, n\} \rightarrow \omega$, the formula

$$\varphi_0(x; \overline{a}_{0, \eta(0)}) \wedge \ldots \wedge \varphi_n(x; \overline{a}_{n, \eta(n)})$$

is consistent.

\end{enumerate}

Given that each set $\widetilde{D}_i$ is infinite and discrete, we can pick pairs $\overline{a}_{i,j} = (a^0_{i,j}, a^1_{i,j})$ such that the open intervals $I(a^0_{i,j}, a^1_{i,j})$ which they define are pairwise disjoint, and each such interval contains two consecutive elements of $\widetilde{D}_i$. For definiteness, we fix elements $c_{i,j} \in \widetilde{D}_i$ for each $i \in \{0, \ldots, n\}$ and each $j < \omega$ such that

\begin{equation}
\label{c_choice}
a^0_{i,j} < c_{i,j} < S_{\widetilde{D}_i}(c_{i,j}) < a^1_{i,j}.
\end{equation}

From the definition of the formulas $\varphi_i(x; \overline{y})$ and the fact that the intervals $I(a^0_{i,j}, a^1_{i,j})$ are pairwise disjoint condition (1) above is immediate.

To show (2), fix some $\eta \, : \, \{0, \ldots, n\} \rightarrow \omega$, and we will show that $$e := c_{0, \eta(0)} + \ldots + c_{n, \eta(n)}$$ satisfies the formula  $\varphi_i(x; \overline{a}_{i, \eta(i)})$ for each $i \in \{0, \ldots, n\}$.

In case $i = 0$, note that by the conclusion of Claim~\ref{gamma_sum}, $$\gamma_{\widetilde{D}_0}(c_{0, \eta(0)}) > c_{1, \eta(1)} + \ldots + c_{n, \eta(n)}$$ so that $$c_{0, \eta(0)} < c_{0, \eta(0)} + \ldots + c_{n, \eta(n)} < S_{\widetilde{D}_0}(c_{0, \eta(0)}).$$ By the inequality (\ref{c_choice}) and the definition of $e$, the element $e$ is in the interval $I(a^0_{0, \eta(0)}, a^1_{0, \eta(0)})$ defined by $\varphi_0(x; \overline{a}_{0, \eta(0)}),$ as we wanted.

If $i \in \{1, \ldots, n\}$, we argue similarly: by the conclusion of Claim~\ref{gamma_sum} again,

$$\sum_{k=0}^{i-1} c_{k, \eta(k)} < e < \sum_{k=0}^{i-2} c_{k, \eta(k)} + S_{\widetilde{D}_{i-1}}(c_{i-1, \eta(i-1)}),$$ and thus, by Lemma~\ref{sums}, the greatest element of $\widetilde{D}_0 + \ldots + \widetilde{D}_{i-1}$ below $e$ is $\sum_{k=0}^{i-1} c_{k, \eta(k)}$; therefore,

\begin{equation}
\label{e-z}
\min \left\{ e - z \, : \, z \in \widetilde{D}_0 + \ldots + \widetilde{D}_{i-1} \textup{ and } z < e \right\} = c_{i, \eta(i)} + \ldots + c_{n, \eta(n)}.
\end{equation}

If $i = n$, then this reduces to $$\min \left\{ e - z \, : \, z \in \widetilde{D}_0 + \ldots + \widetilde{D}_{n-1} \textup{ and } z < e \right\} =  c_{n, \eta(n)},$$

and by (\ref{c_choice}) we have $$a^0_{n, \eta(n)} < \min \left\{ e - z \, : \, z \in \widetilde{D}_0 + \ldots + \widetilde{D}_{n-1} \textup{ and } z < e \right\} < a^1_{n, \eta(n)}$$ so that $e$ satisfies $\varphi_n(x; \overline{a}_{n, \eta(n)})$ as desired.

If $i \in \{1, \ldots, n-1\}$, then once again by Claim~\ref{gamma_sum}, $$\gamma_{\widetilde{D}_i}(c_{i, \eta(i)}) > c_{i+1, \eta(i+1)} + \ldots + c_{n, \eta(n)}$$ so that $$c_{i, \eta(i)} < c_{i, \eta(i)} + c_{i+1, \eta(i+1)} + \ldots + c_{n, \eta(n)} < S_{\widetilde{D}_i}(c_{i, \eta(i)}).$$ By Equation (\ref{e-z}), this implies that

$$c_{i, \eta(i)} <  \min \left\{ e - z \, : \, z \in \widetilde{D}_0 + \ldots + \widetilde{D}_{i-1} \textup{ and } z < e \right\} < S_{\widetilde{D}_i}(c_{i, \eta(i)}).$$

By the inequalities in (\ref{c_choice}), we again deduce that $e$ satisfies $\varphi_i(x; \overline{a}_{i, \eta(i)})$.

Finally, suppose that $\mc{R}$ is densely ordered. Then by $\omega$-saturation, we can pick parameters $\overline{a}_{n+1,j} = (a^0_{n+1,j}, a^1_{n+1,j})$ from $R$ such that for every $j \in \omega$,

$$0 < a^0_{n+1,j} < a^1_{n+1,j} < a^0_{n+1,j+1} < a^1_{n+1,j+1} < \min (\widetilde{D}'_n).$$

By density of $\mathcal{R}$, we can pick elements $c_{n+1, j} \in I(a^0_{n+1,j}, a^1_{n+1,j})$. Now if $\eta \, : \, \{0, \ldots, n+1\} \rightarrow \omega$ is any function, let $$e = c_{0, \eta(0)} + \ldots + c_{n+1, \eta(n+1)}.$$ Arguing just as before, we have that $e$ satisfies $\varphi_i(x; \overline{a}_{i, \eta(i)})$ for every $i \in \{0, \ldots, n+1\}$, and we are done.

\end{proof}

At last, we can prove Theorem~\ref{finite_diff_set}.

\begin{proof} \emph{(of Theorem~\ref{finite_diff_set})} Let $D$ be an infinte definable discrete set such that $D^{(n)}$ is infinite, and we must show that the burden of $\mc{R}$ is at least $n+1$ (or at least $n+2$, in case $\mc{R}$ is densely ordered). By passing to an elementary extension as needed, we assume $\mc{R}$ is $\omega$-saturated.

First we note:

\begin{claim}
There is a definable discrete set $\widetilde{D} \subseteq R$ such that $0 < \widetilde{D}$ and $\widetilde{D}^{(n)}$ is infinite.
\end{claim}

\begin{proof}
If $n = 0$, there is some $a$ such that $(D + a) \cap (0, \infty)$ is infinite\footnote{Note that the existence of such an $a$ is guaranteed by $\omega$-saturation.}, and we can take $\widetilde{D} = (D + a) \cap (0, \infty)$. If $n \geq 1$, first take some bounded interval $(b,c)$ such that $(D \cap (b,c))^{(n)}$ is infinite; then, noting that for any $a \in R$, $$(a + D \cap (b,c))^{(n)} = (D \cap (b,c))^{(n)},$$ we may let $\widetilde{D}$ be a translation of $(D \cap (b,c))^{(n)}$ such that $\widetilde{D} > 0$.
\end{proof}

To prove Theorem~\ref{finite_diff_set}, by Theorem~\ref{descending_prime}, it is sufficient to construct a new sequence of infinite discrete sets $E_0, \dots, E_n$ such that

\begin{enumerate}
\item $0 < E_0$, and
\item for every $i \in \{0, \ldots, n-1\}$, $$0 < E_{i+1} < E'_i.$$
\end{enumerate}

To this end, we will build an $(n+1) \times (n+1)$ matrix of infinite discrete sets $\{E_{i,j} \, : \, 0 \leq i, j \leq n\}$ such that:
\begin{enumerate}
\setcounter{enumi}{2}
\item $E_{0,j} = \widetilde{D}^{(j)}$ for each $j \in \{0, \ldots, n\}$;
\item $E_{i,j} \subseteq E_{i-1,j}$ for all $0 \leq j \leq n$ and all $1 \leq i \leq n$; 
\item If $0 \leq i \leq n$ then $0 < E_{i,j}<E_{i,j-1}'$ for all $n-i +1\leq j \leq n$; and
\item If $0 \leq i \leq n$ and $j < n - i$, then $E_{i, j} = \widetilde{D}^{(j)}$.
\end{enumerate}

The sets $E_{i,j}$ will be defined recursively, starting with the first row (when $i = 0$). For $i=0$, we simply observe that each set $E_{0,j} = \widetilde{D}^{(j)}$ is discrete by Fact~\ref{discrete_image}, and (4), (5), and (6) are trivial. 

Now suppose we have constructed $E_{i,0}, \dots, E_{i,n}$ satisfying (3)-(6), and we show how to define $E_{i+1, 0}, \ldots, E_{i+1,n} $.  We will define the sets $E_{i+1, j}$ by four different cases according to how $j$ relates to $n-i$.

In the first case, when $j \geq n-i+1$, we set $E_{i+1,j}=E_{i,j}$. 

For the second case, we must define $E_{i+1, j}$ when $j = n-i$. As $E_{i,n-i}$ is an infinite discrete set, by $\omega$-saturation of $\mc{R}$ there is a $c \in R$ so that both sets $$F_0=\{x \in E_{i,n-i}: x<c\}$$ and $$F_1=\{x \in E_{i,n-i} :x>c\}$$ are infinite.
Let $$E_{i+1,n-i}=F_0$$ and note that, in case $i > 0$, we have $E_{i+1,n-i+1}<E_{i+1,n-i}'$ by (5) for $j = n - i + 1$ .  

For the third case, when $j = n-i-1$, we set $$E_{i+1, n-i-1}=\{x \in \widetilde{D}^{(n-i-1)} : \gamma_{\widetilde{D}^{(n-i-1)}}(x) \in F_1\}.$$  This is an infinite discrete subset of $\widetilde{D}^{(n-i-1)} = E_{i, n-i-1}$ since $F_1 \subseteq E_{i,n-i} \subseteq \widetilde{D}^{(n-i)}$, and if $a \in E_{i+1, n-i-1}$, then $\gamma_{E_{i+1,n-i-1}}(a)>c$ and hence $E_{i+1, n- i-1}'>E_{i+1, n-i}$.  

Finally, for $j<n-i-1$, we define $E_{i+1,j}=E_{i,j} = \widetilde{D}^{(j)}$.  

Now letting $E_j = E_{n,j}$, we have a sequence of sets satisfying (1) and (2) above, and we are done.
\end{proof}

\section{Fine Structure for Discrete Sets in the Burden $2$ Case}

In this section, we will prove Theorem~\ref{def_in_G}.

Throughout this section $\mc{R}$ will be an expansion of a divisible ordered Abelian group  of burden $2$ which is definably complete and presented in a language $\mc{L}$, and $D$ will always be an infinite discrete set definable in such a structure.  Recall that by Theorem~\ref{finite_diff_set} if $E \subset R$ is any $\mc{R}$-definable discrete set (in particular $D$) then $E'$ is finite.

Theorem \ref{def_in_G} will follow from a result that such a  $D$ must be a finite union of sets which locally resemble arithmetic sequences. To make this more precise, we give the following definition.

\begin{definition}  A discrete set $E$ is called {\em pseudo-arithmetic} if it is infinite and $|E'|=1$.  If $E'=\{\eta\}$ we will  call $E$ {\em$\eta$-pseudo-arithmetic}.  A definable set $X$ is called {\em piecewise pseudo-arithmetic} if it is a finite union of points and definable pseudo-arithmetic sets.
\end{definition}

The bulk of this section is dedicated to proving:

\begin{thm}\label{union-arith}  If $\mc{R}=\langle R; +, <, \dots\rangle$ is a definably complete  expansion of a divisible ordered Abelian group of burden 2 and $D \subseteq R$ is definable and discrete then $D$ is piecewise pseudo-arithmetic.
\end{thm}

In order to establish Theorem~\ref{union-arith} we first prove a weaker result with more assumptions on the discrete set $D$ and then later show how the general result can be deduced from this weaker version.  In order to state our desired weakening of 
Theorem~\ref{union-arith} we need a definition.

\begin{definition}  $D$ is  {\em narrow} if any two elements of $D'$ are Archimedean equivalent.
\end{definition}

We may now state our weaker version of Theorem~\ref{union-arith}.

\begin{prop}\label{weak-union-arith}  If $\mc{R}=\langle R; +, <, \dots\rangle$ is an $\omega$-saturated, definably complete  expansion of a divisible ordered Abelian group of burden 2 and $D$ is definable, discrete, narrow, and bounded below, then $D$ is piecewise pseudo-arithmetic.
\end{prop}

We begin by proving Proposition~\ref{weak-union-arith}.  Then in turn we show how  we may reduce Theorem~\ref{union-arith} to this weaker proposition.  Once we have established Theorem~\ref{union-arith} we then show how our main result,Theorem~\ref{def_in_G}, follows.  We finish the section with an example showing the limits of Theorem~\ref{def_in_G}.

Before we continue we note  that in order to  establish Theorem~\ref{union-arith} we may assume that the structure $\mc{R}$ is $\omega$-saturated.  

\begin{prop}
\label{sat_assumption}
For any infinite cardinal $\kappa$, to prove Proposition~\ref{union-arith} in general, it is sufficient to prove it under the assumption that the structure $\mc{R}$ is $\kappa$-saturated.
\end{prop}

\begin{proof}
Suppose that $\mc{R}=\langle R; +, <, \dots\rangle$ is any definably complete  expansion of a divisible ordered Abelian group (not necessarily $\kappa$-saturated) of burden 2 and $D$ is an infinite definable discrete set.  Let $\mc{R}_1$ be a $\kappa$-saturated elementary extension of $\mc{R}$, and let $D_1 \subseteq R_1$ be the subset definable by the same formula that defines $D$ in $R$. Then $D_1$ is also discrete. Applying  Proposition~\ref{union-arith} for $\kappa$-saturated structures, we conclude that $D_1$ is piecewise pseudo-arithmetic.
Thus we find formulas $\varphi_1(x, \ob{a}_1), \dots, \varphi_m(x, \ob{a}_m)$ and elements $b_1, \dots, b_l \in R_1$ so that 
$D=E_1 \cup \dots \cup E_m \cup \{b_1, \dots, b_l\}$ where each $E_i$ is a pseudo-arithmetic set defined by $\varphi_i(x, \ob{a}_i)$.  But then we can find $\ob{a}_1^*, \dots \ob{a}_m^* \in R$ and $c_1, \dots c_l \in R$ so that 
$D=E_1^*\cup\dots \cup E_m^* \cup \{c_1, \dots, c_l\}$ where each $E_i^*$  is a discrete subset of $R$ defined by $\varphi_i(x, \ob{a}_i^*)$ with $|(E_i^*)'|=1$.  Not all of the $E_i^*$ are necessarily pseudo-arithmetic as some may be finite, but nonetheless we have written $D$ as union of finitely many points and pseudo-arithmetic sets as desired. 
\end{proof}

Note that Proposition~\ref{weak-union-arith} is only an intermediate step in establishing Theorem~\ref{union-arith}, and thus we include the hypothesis that $\mc{R}$ is $\omega$-saturated in the statement of Proposition~\ref{weak-union-arith}.


Finally the following definition will occur throughout the following sections:

\begin{definition} \label{dconvex}
We will say that $X \subseteq D$ is {\em $D$-convex} if whenever $a,b \in D \cap X$ and $a<c<b$ for some $c \in D$ then $c \in X$.
\end{definition}

\subsection{Establishing Proposition~\ref{weak-union-arith}}

In this subsection we prove Proposition~\ref{weak-union-arith}.   Thus in this subsection we will maintain:

\begin{ass}
\label{3.1}
  $D$ is a definable discrete set which is narrow and bounded below, and $\mc{R}$ is $\omega$-saturated.
\end{ass}

In order to prove Proposition~\ref{weak-union-arith} we first need to establish a sequence of preliminary results.  The first of these is a statement describing definable subsets of $D$.

\begin{prop}\label{struct-def} Let $E \subseteq D$ be definable.  Then there are $M \in \NN$ and finitely many closed intervals or closed rays $C_1, \dots, C_n$ in $R$ such that:
\begin{enumerate}
\item $E\subseteq(C_1 \cap D)  \cup \dots \cup (C_n \cap D)$;
\item The $C_i$ are pairwise disjoint;
\item If $C_i = [b,b] = \{b\}$, then $b \in E$;
\item If $C_i$ is an infinite closed interval, then $C_i \cap E$ is infinite;
\item If $C_i$ is an infinite interval and $C_i$ is bounded above, then $C_i$ is of the form $[a_i, b_i]$ with $b_i \in E$ and $a_i \in E$, and otherwise it is of the form $[a_i, \infty)$ with $a_i \in E$;
\item If $C_i$ is an infinite interval and $Z$ is a $\ZZ$-chain with $Z \subseteq C_i$ then $E$ is both cofinal and coinitial in $Z$, and moreover for all $b \in Z$, $\{b, S_D(b), \dots S_D^M(b)\} \cap E \not= \emptyset$;
\item If $C_i$ is an infinite interval of the form $[a_i,b_i]$ or $[a_i, \infty)$ then $E$ is cofinal in  $\mc{Z}(a_i)$, and moreover for all $b \in \mc{Z}(a_i)$ with $b \geq a_i$, \[\{b, S_D(b), \dots, S_D^M(b)\} \cap E \not= \emptyset;\] 
\item If $C_i$ is an infinite interval of the form $[a_i, b_i]$  then $E$ is coinitial in $\mc{Z}(b_i)$, and moreover for all $a \in \mc{Z}(b_i)$ with $a \leq b_i$,
 \[\{a, S_D^{-1}(a), \dots, S_D^{-M}(a)\} \cap E \not= \emptyset.\]
\end{enumerate}
\end{prop} 

We begin with an essential lemma and some useful corollaries.

\begin{lem}\label{uniform}  Suppose that $\{D_a: a \in X\}$ is a definable family of subsets of $D$ and $\delta = \max(D')$.  Then there is an $l \in \omega$ such that for any $a \in X$, the set $\{d \in D_a \, : \, \gamma_{D_a}(d) \geq l \cdot \delta\}$ is finite.

\end{lem}

\begin{proof} Suppose that no such $l$ exists.  Then by compactness we find $a \in X$ such that there are infinitely many $d \in D_a$ with $\gamma_{D_a}(d) \gg \delta$.  Let $\{d_i \, : \, i \in \omega\}$ list the first $\omega$ elements of $D$.  By $\omega$-saturation of $\mc{R}$, there is an element $\varepsilon \in R$ such that
  
  \begin{enumerate}
  \item for every $i \in \omega$, $d_i < \varepsilon$; and
  \item there are infinitely many $d \in D_a$ such that $\gamma_{D_a}(d) > \varepsilon$.
  \end{enumerate}
Let 
$F=\{x \in D_a : \gamma_{D_a}(x)\geq\varepsilon\}$ and notice that $F' \geq \varepsilon$. Thus $F$ and $D \cap [0, \varepsilon)$ are infinite discrete sets such that $D \cap [0,\varepsilon]<F'$, violating Theorem~\ref{descending_prime}.
\end{proof}

If $E \subseteq D$ is definable, then we can apply Lemma~\ref{uniform} to the family of sets $D_a = E$ to immediately deduce the following:

\begin{cor}\label{fin-big}  If $E \subseteq D$ is definable and $\delta = \max(D')$, then there can be only finitely many $b \in E$ with $\gamma_E(b) \gg \delta$.
\end{cor}



Let $\equiv$ be the convex equivalence relation on $D$  so that  $a \equiv b$ for $a, b \in D$ if and only if 
$|a-b|<l \cdot \max{(D')}$ for some $l \in \omega$.   Notice as $D$ is narrow we also have $a \equiv b$ if and only if $|a-b|<l \cdot\min{(D')}$ for some $l \in \omega$.  Or to state things more simply, $a \equiv b$ if and only if $a$ and $b$ lie on the same $\ZZ$-chain of $D$.    By saturation $\sfrac{D}{\equiv}$ is a dense linear ordering.

\begin{cor}\label{convexcomp}  If $E \subseteq D$ is definable then $\sfrac{E}{\equiv}$ has finitely many convex components in $\sfrac{D}{\equiv}$.
\end{cor}

\begin{proof}  Suppose otherwise and let $C_i$ for $i \in \omega$ be distinct convex components of $\sfrac{E}{\equiv}$ in $\sfrac{D}{\equiv}$.  As $\mathcal{R}$ is assumed to be $\omega$-saturated we may assume by compactness that $C_i<C_{i+1}$ for all $i \in \omega$.  We may find $a_i \in D \setminus E$ for $i \in \omega$ so that $C_i<\left(\sfrac{a_i}{\equiv}\right)<C_{i+1}$ and so that $a_i \not\equiv e$ for any $e \in E$.

Now let $b_i =\max\{e \in E : e<a_i\}$ for $i \in \omega$.  As $C_i<\left(\sfrac{a_i}{\equiv}\right)$ and  $\mathcal{R}$ is definably complete,  $b_i$ exists; and as $E$ is closed (by Fact~\ref{closedd}), $b_i \in E$.  
Note that $b_i$ is not maximal in $E$ for any $i \in \omega$ as $\left(\sfrac{b_i}{\equiv}\right)<C_{i+1}$.   We claim that $\gamma_E(b_i)>l \cdot \max{(D')}$ for all $i, l \in \omega$.   Suppose otherwise and that $\gamma_E(b_i) \leq l \cdot \max{(D)}$ for some $i,l \in \omega$.   But by choice of $b_i$ we must have that $S_E(b_i)>a_i$.  Hence $|a_i-b_i| \leq l \cdot \max{(D')}$ and so $a_i \equiv b_i$  but this violates our choice of $a_i$.  Thus for all $i, l \in \omega$ we have that $\gamma_E(b_i)>l \cdot \max{(D')}$.  But this violates  Corollary \ref{fin-big}.


\end{proof}

\begin{cor}\label{gaps}  Let $E \subseteq D$ be definable.  Then there is $M_0(E) \in \omega$ so that:

\begin{enumerate}
\item If $E$ is cofinal in a $\ZZ$-chain $Z$ then 
\[\{b, S_D(b), \dots, S_D^{M_0(E)}(b)\} \cap E \not=\emptyset\] for all sufficiently large $b \in Z$;
\item If $E$ is coinitial in a $\Z$-chain $Z$ then
  \[\{b, S_D(b), \dots, S_D^{-M_0(E)}(b)\} \cap E \not=\emptyset\] for all sufficiently small $b \in Z$;
\item If $E$ is both cofinal and coinitial in a $\Z$-chain $Z$ then
 \[\{b, S_D(b), \dots, S_D^{M_0(E)}(b)\} \cap E \not=\emptyset\] for all $b \in Z$.
\end{enumerate}
\end{cor}

\begin{proof} First suppose that (1) fails.  Let $N \in \omega$.  For $i \in \omega$ we may find $a_i \in D$ all lying on a single $\ZZ$-chain so that 
\[\{S_D^{-N}(a_i), \dots, S^{-1}_D(a_i), a_i, S_D(a_i), \dots, S^N_D(a_i)\}\cap E = \emptyset\] and $b_i \in E$ with $a_i<b_i<a_{i+1}$ for all $i \in \omega$.

Then by compactness we find $a_i^* \in D$ for $i \in \omega$ so that $\{S^M_D(a_i^*) : M \in \ZZ\} \cap E = \emptyset$ and $b_i^* \in E$ so that $a_i^*<b_i^*<a_{i+1}^*$.  For each $i \in \omega$ we have that $|a_i^*-c|>l \cdot \min{(D')}$ for all $l \in \omega$  and any $c \in E$.  As $D$ is narrow $a_i^* \not\equiv c$ for any $c \in E$.  Thus $\left(\sfrac{a_i^*}{\equiv}\right)<\left( \sfrac{b_i^*}{\equiv}\right)<\left(\sfrac{a_{i+1}^*}{\equiv}\right)$ and furthermore $\left(\sfrac{a_i^*}{\equiv}\right) \notin \sfrac{E}{\equiv}$ but $\left(\sfrac{b_i^*}{\equiv}\right) \in \sfrac{E}{\equiv}$ for all $i \in \omega$.  In particular   each $\left(\sfrac{b_i^*}{\equiv}\right)$ is in a distinct convex component of $\sfrac{E}{\equiv}$ in 
$\sfrac{D}{\equiv}$ violating
Corollary~\ref{convexcomp}

The proof of (2) is identical to that of (1) except that for each $N \in \omega$ we find $c_i \in E$ and $a_i \in D$ for $i \in \omega$ so that \[\{S_D^{-N}(a_i), \dots, S^{-1}_D(a_i), a_i, S_D(a_i), \dots, S^N_D(a_i)\}\cap E = \emptyset\] with $a_i>c_i>a_{i+1}$.

Now assume that (3) fails.  As (1) and (2) hold, for each $N \in \omega$ we can find  pairwise distinct $\ZZ$-chains $Z_i$ of $D$ and $
a_i \in Z_i$ for $i \in \omega$ so that \[\{S_D^{-N}(a_i), \dots, S^{-1}_D(a_i), a_i, S_D(a_i), \dots, S^N_D(a_i)\}\cap E = \emptyset\] together with $b_i, c_i \in E \cap Z_i$ so that $c_i<a_i<b_i$.   We can assume that either the $Z_i$ are ordered so that $Z_i<Z_{i+1}$ for all $i \in \omega$ or that $Z_i>Z_{i+1}$ for all $i \in \omega$.  In the case that $Z_i<Z_{i+1}$ argue with $a_i, b_i$ exactly as in the proof of (1).  In the case that $Z_i>Z_{i+1}$ argue with $a_i, c_i$ exactly as in the proof of (2). 
\end{proof}

We may now prove Proposition~\ref{struct-def}.
Recall that we always assume that $D$ is bounded below,  thus any interval or ray $C_i$ occurring in the statement of Proposition~\ref{struct-def} may be assumed to be bounded below.

\begin{proof}  For the purposes of this proof let us call a definable $E \subseteq D$ {\em neat} if we may find $M \in \NN$ and $C_1, \dots, C_n$ closed intervals or rays satisfying condition (1)-(8) of Proposition~\ref{struct-def}.

Notice that for clauses (6), (7), and (8) in the statement of Proposition~\ref{struct-def} the ``moreover'' portion will follow from the initial statements and Corollary \ref{gaps}.  Hence to establish that a set is neat we need only produce the intervals or rays $C_1, \dots, C_n$ witnessing (1)-(8) in the definition of neat without the ``moreover'' clauses.

It is immediate that if $E=E_1 \cup \dots \cup E_l$ where the $E_i$ are definable, pairwise disjoint, and neat then $E$ is neat.

Let $E \subseteq D$ be definable.  By Corollary \ref{convexcomp}, $\sfrac{E}{\equiv}$ has finitely many convex components.  Let $H_j$ for $1 \leq j \leq k$ be the  convex components of $\sfrac{E}{\equiv}$.  We may pick $e_j,f_j \in D \cup \{\infty\}$ so that $H_j = \sfrac{E \cap [e_j,f_j)}{\equiv}$.  Setting $E_j=[e_j,f_j) \cap E$, to establish that $E$ is neat it suffices to show that each $E_j$ is neat.

Thus fix $1 \leq j \leq k$.  If $E_j$ is finite then the result is immediate, hence we may assume that $E_j$ is infinite.  Let $a =\min(E_j)$, which exists as $E_j$ is bounded below.  If $E_j$ is bounded above let $b=\max(E_j)$.  Let $Z$ be any $\ZZ$-chain so that $E_j \cap Z\not=\emptyset$ and $\mc{Z}(a)<Z<\mc{Z}(b)$ (if $b$ exists).  We claim that $E_j$ must be both cofinal and coinitial in $Z$.  Suppose otherwise, then for any $N \in \omega$ we can find $c \in Z$ so that $\{S^{l}_D(c) : -N \leq l \leq N\} \cap E_j =\emptyset$.  By compactness we can then find another $\ZZ$-chain $Z_0$ so that $\mc{Z}(a)<Z_0<\mc{Z}(b)$ (if $b$ exists) and $Z_0 \cap E_j= \emptyset$. But this violates the fact that $\sfrac{E_j}{\equiv}$ is convex in $\sfrac{D}{\equiv}$.  Hence $E_j$ is cofinal and coinitial in $Z$.  Similarly we can conclude that $E_j$ is confinal in $\mc{Z}(a)$ and coinitial in $\mc{Z}(b)$ (if $b$ exists).  Thus  $C_1=[a,b]$ if $b$ exists, or $C_1=[a, \infty)$ otherwise  is a closed interval or ray witnessing that $E_j$ is neat.  
\end{proof}

To continue our proof of Proposition~\ref{weak-union-arith} we proceed with a combinatorial analysis of $\Z$-chains in $D$ much like that in \cite{DG} to establish that $\Z$-chains in $D$ are approximately periodic.  To make this precise we need a definition.

\begin{definition}
\label{periodic_chains}
 Suppose that $\tau$ is an infinite sequence $\{a_i \, : \, i \in I\}$ of elements from the set $D'$ where the index set $I$ is either $\Z$, $\omega$, or $\omega^*$ (that is, $\omega$ with the reverse ordering). In case $I = \omega^*$, we use $i+k$ to denote the $k$-th successor of $i$ if such an element exists, thus whenever we write $i+k$ we are presuming the $k$-th successor of $i$ exists.
 
 \begin{enumerate}
 \item The sequence $\tau$ is \emph{$m$-periodic} if whenever $i$ and $i+m$ belong to $I$, we have $a_{i+m} = a_i$.
 \item If $I = \omega$, the sequence $\tau$ is \emph{eventually $m$-periodic} if there is some $N \in I$ such that for every $k \geq N$, we have $a_{k+m} = a_k$.
 \item If $I = \omega^*$, the  $\tau$ is \emph{eventually $m$-periodic} if there is some $N \in \omega^*$ such that for every $k \leq N$, we have $a_{k+m} = a_k$.
 \item If $I = \Z$, the sequence $\tau$ is \emph{eventually $m$-periodic} if \textbf{both} of the subsequences $\{a_i \, : \, i \geq 0\}$ and $\{a_i \, : \, i \leq 0\}$ are eventually $m$-periodic, according to (2) and (3) above.
 \item If $I = \omega$ or $\omega^*$ and $\sigma \in (D')^m$, the sequence $\tau$ is \emph{$\sigma$-periodic} if it is $m$-periodic and the first $m$ elements of $\tau$ (or the last $m$ elements, in case $I = \omega^*$) are a cyclic permutation of $\sigma$. The notion of \emph{eventually $\sigma$-periodic} are defined analogously.
 \item The sequence $\tau$ is \emph{(eventually) periodic} if it is (eventually) $m$-periodic for some $m \in \omega \setminus \{0\}$.
 \item A $\ZZ$-chain $Z$ from $D$ is \emph{(eventually) $m$-periodic} if for some $a \in Z$, the corresponding sequence of elements $\langle \gamma_D^i(a) \, : \, i \in \Z\rangle$ is (eventually) $m$-periodic. Similarly for subsets of $D$ of the form $\mc{Z}_{\geq}(b)$ or $\mc{Z}_{\leq}(b)$.

\end{enumerate}

\end{definition}

For the next Proposition and its consequences the following definition is convenient.

\begin{definition}  If $E$ is an infinite definable discrete set and $\{\simga_1, \dots, \simga_k\}$ is a finite set of finite sequences from $E'$ we call $\{\sigma_1, \dots, \sigma_k\}$ a {\em characteristic set} for $E$ if  every $\omega$-chain of $E$ and every $\omega^*$-chain of $E$ is eventually $\simga_i$-periodic for some $1 \leq i \leq k$.
\end{definition}

Our next goal is to show:

\begin{prop}\label{periodicconc}  $D$ has a characteristic set.
\end{prop}

In order to establish Proposition~\ref{periodicconc} we need some preliminary definitions and lemmas.

\begin{definition} For $a, b, c, d \in D$ such that $a<b$ and $c < d$, we define $(a,b) \sim (c,d)$ if and only if $b-a=d-c$ and for all $e$ in the interval $(a,b)$ we have that $e \in D$ if and only if $e+c-a \in D$.  For any pair $(a,b) \in D^2$ such that $a<b$, let $P_{(a,b)}=\{c \in D : (a,b) \sim (c, c+(b-a))\}$.
\end{definition}

Roughly speaking $(a,b) \sim (c,d)$ if these two intervals are ``isomorphic'' as far as $D$ is concerned.

\begin{definition} For a finite sequence $\sigma=\langle c_1, \dots, c_m\rangle$ of elements from  $D'$, let $P_{\sigma}$ be the set of all $d \in D$ 
so that $S^l(d)-S^{l-1}(d)=c_l$ for all $ l \in \{1, \ldots, m\}$. 
\end{definition}

Notice that if ${\simga}$ is any finite sequence from $D'$ with $P_{{\simga}}$ non-empty then 
$P_{{\sigma}}$ is equal to $P_{(a,b)}$ for some $a, b \in D$.  In particular the family of all $P_{{\sigma}}$ is part of a uniformly definable family of subsets of $D$.

\begin{lem}\label{recurbound} There is some fixed $m_0 \in \omega$ such that for any $\omega$-chain $C$ in $D$ and any $k \in \omega$, there are at most $m_0$ sequences ${\simga}$ of length $k$ so that $P_{{\sigma}}$ is not bounded above in $C$. Similarly, there is a fixed $m_1 \in \omega$ such that for every $\omega^*$-chain $C$ in $D$ and any $k \in \omega$, there are at most $m_1$ sequences ${\simga}$ of length $k$ so that $P_{{\sigma}}$ is not bounded below in $C$.
\end{lem}

\begin{proof} Say $D' = \{\delta_1, \ldots, \delta_n\}$ where the $\delta_i$ are listed in increasing order. First we will give the proof of the existence of $m_0$. By Lemma \ref{uniform}, there is some $l_0 \in \omega$ such that for any $a, b \in D$, there are only finitely many $c \in P_{(a,b)}$ with $\gamma_{P_{(a,b)}}(c)\geq l_0 \cdot \delta_n$. Since $\delta_1$ is Archimedean equivalent to $\delta_n$, there is some $l \in \omega$ such that $l \cdot \delta_1 \geq l_0 \cdot \delta_n$, and hence:

\medskip
\textit{$(*)$ For any $a, b \in D$, there are only finitely many $c \in P_{(a,b)}$ with $\gamma_{P_{(a,b)}}(c)\geq l \cdot \delta_1$.}
\medskip

We claim that $m_0 = l$ works.   Fix an $\omega$-chain $C$ and a natural number $k$ and suppose that there are $m$ distinct sequences ${\sigma}_1 \dots, {\sigma}_m$ of length $k$ from $D'$ so that no $P_{{\sigma}_i}$ is bounded above in $C$.  If there is a $\sigma_i$ such that the set $$\{ a \in C \, : \, P_{{\sigma}_i} \cap \{a, S_D(a), \ldots, S_D^{l-1}(a)\} = \emptyset\}$$ is cofinal in $C$, then $P'_{\sigma_i}$ contains infinitely many elements which are greater than or equal to $l \cdot \delta_1$, contradicting statement $(*)$ above. Therefore for all sufficiently large $a \in C$, each set $P_{{\simga}_i}$ must intersect $\{a, S_D(a), \dots S_D^{l-1}(a)\}$. Since the sets $P_{\sigma_1}, \ldots, P_{\sigma_m}$ are pairwise disjoint, this implies that $m \leq l$, as we wanted.

The same proof, \emph{mutatis mutandis}, shows the existence of $m_1$.

\end{proof}

At this point, we can repeat some arguments from our paper \cite{DG} (specifically, Lemma 2.33 and Proposition 2.35) to conclude that every $\omega$-chain is eventually periodic. In the interest of making the current paper self-contained, we now present a simplified version of the combinatorial argument found there. Suppose that $\tau$ is an infinite sequence $a_0, a_1, \ldots$ of elements from the set $D' = \{\delta_1, \ldots, \delta_n\}$, and recall the definition of ``eventually $m$-periodic'' from Definition~\ref{periodic_chains} above. Say that a finite sequence $\sigma$ from $D'$ is \emph{infinitely recurring in $\tau$} in case it occurs infinitely often as a subsequence $a_i, a_{i+1}, \ldots, a_{i+k}$ of $\tau$. Let $f_\tau \, : \, \omega \rightarrow \omega$ be the function defined by the rule that $f_\tau(k)$ is the number of infinitely-recurring sequences of length $k$ in $\tau$. 

The crucial fact about $f_\tau$ is the following:

\begin{lem}
\label{tau_periodic}
If $M \in \omega$ and $f_\tau(k) \leq M$ for every $k \in \omega$, then the sequence $\tau$ is eventually $m$-periodic for some $m \leq M$.
\end{lem}

\begin{proof}
Suppose that $f_\tau(k) \leq M$ for every $k \in \omega$. Note that since $D'$ is finite, it follows from the Pigeonhole Principle that every infinitely-recurring sequence of length $k$ in $\tau$ can be extended to at least one infinitely-recurring sequence of length $k+1$ in $\tau$. Hence $$f_\tau(k) \leq f_{\tau}(k+1),$$ that is, the function $f_\tau$ is nondecreasing. By the hypothesis that values of  $f_\tau$ are bounded above by some $M$, there is $k \in \omega$ for which $f_\tau(k) = f_\tau(k+1)$. 

\begin{claim}
Without loss of generality, there is $k \in \omega$ such that $f_\tau(k) = f_{\tau}(k+1)$ and furthermore every subsequence of $\tau$ of length $k$ or $k+1$ is infinitely recurring.
\end{claim}

\begin{proof}
First pick some $k$ such that $f_\tau(k) = f_\tau(k+1)$, which exists by the argument just above. If $\sigma$ is a subsequence of $\tau$ which is \emph{not} infinitely recurring, then there is some point $p_\sigma \in \omega$ such that $\sigma$ never occurs past $p_\sigma$. Let $p$ be the maximum of $p_\simga$ for all subsequences $\sigma$ of length $k$ or $k+1$ which occur only finitely often in $\tau$ (of which there can be only finitely many, since $D'$ is finite). Then every subsequence of length $k$ or $k+1$ which occurs past the point $p$ must be infinitely recurring.
\end{proof}

Now fix some $k$ as in the conclusion of the Claim above. Hence if $\sigma$ is any length-$k$ subsequence of $\tau$, it has a unique length-$(k+1)$ end extension $\sigma^+$ which occurs as a subsequence of $\tau$, and we define $\sigma'$ to be the sequence formed by the last $k$ elements of $\sigma^+$. Let $\sigma$ be the first $k$ elements of $\tau$ and form the sequence $$\sigma, \sigma', \sigma'', \ldots, \sigma^{(i)}, \ldots$$ by recursively applying this operation. By the Pigeonhole Principle, there must be $i < j$ such that $\sigma^{(i)} = \sigma^{(j)}$, and we may pick $j$ so that $j-i \leq f_\tau(k) \leq M$. Thus $\tau$ is $(j-i)$-periodic past the first instance of $\sigma^{(i)}$, and we are done.

\end{proof}

From the previous Lemma, we can quickly deduce Proposition~\ref{periodicconc}:

\begin{proof}
 We will show that there are finitely many finite sequences $\sigma_1, \ldots, \sigma_k$ such that any $\omega$-chain of $D$ is eventually $\sigma_i$-periodic for some $i \in \{1, \ldots, k\}$, and the same argument, with only notational changes, applies to $\omega^*$-chains of $D$, yielding the conclusion of Proposition~\ref{periodicconc}.

So fix some $\omega$-chain $C$ of $D$. Pick $a\in C$ arbitrarily and let $\tau$ be the $\omega$-chain $(\gamma_D(a), \gamma_D(S_D(a)) \dots \gamma_D(S_D^l(a)), \dots)$. By Lemma~\ref{recurbound}, $f_\tau(n) \leq m_0$ for all $n \in \omega$. Applying Lemma~\ref{tau_periodic} we conclude that $\tau$, and hence $C$, is eventually $m$-periodic for some $m \leq m_0$. Since $D'$ is finite, there are only finitely many sequences of length at most $m_0$ from $D'$, and the conclusion of Proposition~\ref{periodicconc} follows.
\end{proof}

We need one final intermediate proposition elucidating the structure of $D$ which will allow us to deduce Proposition~\ref{weak-union-arith}.  In order to state this result we need a definition.

\begin{definition} Let $\sigma$ be any finite sequence of elements of $D'$. A closed interval $I \subseteq R^{\geq 0}$ is a {\em $\sigma$-interval} if:

\begin{enumerate}

\item $I=[a,b]$ or $[a, \infty)$ with $a,b \in D$.
\item $I \cap D$ is infinite.
\item If $Z$ is a $\ZZ$-chain and $Z \subset I $ then $Z$ is $\sigma$-periodic.
\item If $I=[a,b]$ or $[a, \infty)$ then $\mc{Z}_{\geq}(a)$ is $\sigma$-periodic.
\item If $I=[a,b]$ then $\mc{Z}_{\leq}(b)$ is $\sigma$-periodic.

\end{enumerate}
\end{definition}

Given this definition we aim to prove:

\begin{prop}\label{disc-struct}  Assume $\{\sigma_1, \dots, \sigma_k\}$ is a characteristic set for $D$.   There is a finite collection of pairwise disjoint intervals $\mc{C}$ so that each $I \in \mc{C}$ is a $\sigma_i$-interval for some $1 \leq i \leq k$ and $D \setminus \bigcup_{I \in \mc{C}}I$ is finite.
\end{prop}

Fix  $\{\sigma_1, \dots, \sigma_k\}$, a characteristic set for $D$.

To establish Proposition~\ref{disc-struct} we need a key preliminary Lemma:
\begin{lem}\label{almost_all_periodic}
All but finitely many $\ZZ$-chains in $D$ are periodic.
\end{lem}

\begin{proof}
Suppose toward a  contradiction that there are infinitely many non-periodic $\ZZ$-chains in $D$. If there are sequences in $\{\simga_1, \ldots, \sigma_k\}$ of different lengths, then if $m$ is a common multiple of the lengths of all the $\simga_i$'s, we may replace each $\sigma_i$ with a suitable number of concatenations with itself which has length $m$, and thus without loss of generality each $\sigma_i$ has the same length $m$. There is also no harm in assuming that the set $\{\simga_1, \ldots, \sigma_k\}$ is closed under cyclic permutations.

Fix some $\ZZ$-chain $Z$ which is not periodic. Then there is some $a \in Z$ such that $\mc{Z}_{\geq}(a)$ is not periodic. By Proposition~\ref{periodicconc}, $\mc{Z}_{\geq}(a)$ is eventually periodic. We may assume that $a$ is maximal in the sense that if $a' \in Z$ and $a' > a$, then $\mc{Z}_{\geq}(a')$ is periodic. Let $\tau$ be the $\omega$-chain $\gamma_D(a), \gamma_D(S_D(a)), \gamma_D(S_D^2(a)), \ldots$ of successive elements of $D'$ beginning at $\gamma_D(a)$. Then there is some $d_Z \in D'$ and some $i_Z \in \{1, \ldots, k\}$ such that $\tau = d_Z \sigma_{i_Z}^\omega$ (the concatenation of $d_Z$ followed by $\omega$ copies of $\sigma_{i_Z}$), since by our maximality assumption on $a$, the $\omega$-chain of successive elements of $D'$ which begins at $\gamma_D(S_D(a))$ must be periodic.

By the assumption that there are infinitely many non-periodic $\ZZ$-chains and the Pigeonhole Principle, there is a single $d \in D'$ and a single $i \in \{1, \ldots, k\}$ such that there are infinitely many non-periodic $\ZZ$-chains $Z$ such that $d_Z = d$ and $i_Z = i$. Let $\sigma_i = \sigma_i^- e$ where $\sigma_i^-$ is the prefix consisting of the first $m-1$ elements of $\sigma_i$ and $e$ is the final element of $\sigma_i$. Then $e \neq d$, since otherwise

$$d \sigma_i^\omega = d \sigma_i^- e \sigma_i^- e \ldots$$

would be an $m$-periodic sequence, contrary to our assumption.

Now let $\sigma = d \sigma_i$. Since $d \neq e$, the sequence $\sigma$ cannot be a subsequence of any $m$-periodic $\omega$-chain. Thus there are infinitely many $a \in P_{\sigma}$ such that $a$ is the last element of its $\ZZ$-chain which intersects $P_{\sigma}$. Therefore $P_{\sigma}$ is an infinite definable discrete subset of $D$ and there are infinitely many elements $a \in P_{\sigma}$ such that $S_{P_\sigma}(a) - a \gg D'$, but this contradicts Corollary \ref{fin-big}.

\end{proof}

 \begin{proof}
 \emph{(of Proposition~~\ref{disc-struct})}

By Proposition \ref{periodicconc}, for every narrow definable discrete set $D$ which is bounded below, there is some $k \in \omega$ and finite sequences $\sigma_1, \dots, \sigma_k$ from $D'$ so that $\{\sigma_1, \dots, \sigma_k\}$ is a characteristic set for $D$.  Let $k(D)$ be the minimum number such that there is a characteristic set $\{\sigma_1, \ldots, \sigma_{k(D)}\}$ for $D$. Without changing the number $k(D)$ of sequences in such a set, we may further assume that all the sequences $\sigma_i$ have the same length (arguing as in the first paragraph of the proof of Lemma~\ref{almost_all_periodic}). We also note that by minimality of $k(D)$, no two sequences $\sigma_i$ and $\sigma_j$ for $1 \leq i < j \leq k(D)$ are cyclic permutations of one another. In case $D$ is finite, we define $k(D) = 0$, and whenever $D$ is infinite, $k(D) > 0$.

Our proof is by induction on $k(D)$.  We assume that the conclusion of Proposition~\ref{disc-struct} holds for all  narrow discrete sets $E$ which are bounded below and for which $k(E) < k(D)$.  We prove the conclusion of Proposition \ref{disc-struct} for $D$. Note that the base case is when $k(D) = 0$, which holds if and only if $D$ is finite, in which case the conclusion of Proposition~\ref{disc-struct} is trivial.

Now suppose that $k(D) > 0$. Fix a a characterisitc set$\{\sigma_1, \ldots, \sigma_{k(D)}\}$ for $D$ with all $\sigma_i$ of the same length, none of which is a cyclic permutation of the other. We apply Proposition \ref{struct-def} to the subset of $D$ defined by $P_{\sigma_1}$ to obtain a finite collection of closed, pairwise-disjoint intervals $\mc{C}_0 = \{C_1, \ldots, C_n\}$ such that $P_{\sigma_1} \subseteq \bigcup{1 \leq i \leq n} C_i$ and for each $i \in \{1, \ldots, n\}$,

\begin{enumerate}
\item $P_{\sigma_1}$ is cofinal and coninitial in every $\Z$-chain contained in $D \cap C_i$,
\item $P_{\sigma_1}$ is cofinal in every $\omega$-chain contained in $D \cap C_i$, and
\item $P_{\sigma_1}$ is coinitial in every $\omega^*$-chain contained in $D \cap C_i$.
\end{enumerate}

Since the $\Z$-, $\omega$-, and $\omega^*$-chains in (1), (2), and (3) must be eventually $\sigma_i$-periodic for some $i$, and by our assumptions that the $\sigma_i$'s are all of the same length and not cyclic permutations of one another, it follows that the chains in (1), (2), and (3) are all eventually $\sigma_1$-periodic.

The conclusion of Proposition \ref{struct-def} allows for some intervals $C_i$ which intersect $P_{\sigma_1}$ in only one point; by discarding such $C_i$, we may further assume that every set $C_i \cap P_{\sigma_1}$ is infinite, and that

\begin{enumerate}
\setcounter{enumi}{3}
\item $P_{\sigma_1} \setminus \bigcup_{C_i \in \mc{C}_0} C_i$ is finite.
\end{enumerate}


The intervals in $\mc{C}_0$ may not be $\simga_1$-intervals for either one of two reasons.  First, it may be the case that for some $W$, an  eventually $\sigma_1$-periodic but not $\sigma_1$-periodic $\Z$-chain of $D$, we have that $W \subseteq I$ for some $I \in \mc{C}_0$.  But by Lemma \ref{almost_all_periodic} there are only finitely many such chains.  Thus we can rectify this problem by partitioning the elements of $\mc{C}_0$ into subintervals so that no such $W$ is contained in any of the intervals.  Secondly  for some $J \in \mc{C}_0$ with $J=[a,b]$ or $J=[a, \infty)$ it may be the case that either $Z_{\geq}(a)$ or $Z_{\leq}(b)$ are only eventually $\sigma_1$-periodic and not $\sigma_1$-periodic.  But this is easily fixed by replacing $a$ by a suitably chosen $S^l(a)$ or replacing $b$ by a suitably chosen $S^{-l}(b)$.   With these adjustments $\mc{C}_0$ is a finite collection of $\sigma_1$-intervals.  

 
 Next let $D_1, \dots D_r$ be the finitely many  infinite $D$-convex components of $D \setminus \bigcup_{C_i \in \mc{C}_0} C_i$ (we may ignore any finite $D$-convex components). For each $i \in \{1, \ldots, r\}$, any $\omega$-chain of $D_i$ is also an $\omega$-chain of $D$, and by construction it is not eventually $\sigma_1$-periodic, hence it is $\sigma_j$-periodic for some $j \in \{2, \ldots, k(D)\}$, and likewise for $\omega^*$-chains; thus $\{\sigma_2, \dots, \sigma_{k(D)}\}$ is a characteristic set for each $D_i$, and so $k(D_i) < k(D)$. By the induction hypothesis we may find finite collections of intervals $\mc{C}_i$ for $1 \leq i \leq r$ so that each $I \in \mc{C}_i$ is a $\sigma_j$-interval for some $2 \leq j \leq k(D)$ and so that $D_i \setminus \bigcup_{I \in \mc{C}_i}I$ is finite.  Then setting $\mc{C}=\bigcup_{i =0}^r\mc{C}_i$ yields the desired collection of intervals.

\end{proof}

Now at last we are in a position to prove Proposition~\ref{weak-union-arith}.

 \begin{proof}  Let $D$ be definable, discrete, narrow, and bounded below.    By Proposition \ref{disc-struct}, there are finitely many finite sequences $\sigma_1, \ldots, \sigma_k$ from $D'$ such that $D$ is the union of a finite set $F$ plus finitely many sets of the form $D \cap I$ where $I$ is a $\sigma$-interval for some $\sigma \in \{\sigma_1, \ldots, \sigma_k\}$. Without loss of generality we may assume $D = D \cap I$ for some $\sigma$-interval $I$. 
 
 Also without loss of generality, $\sigma$ is of minimal length -- that is, for any finite sequence $\sigma'$ from $D'$ which is shorter than $\sigma$, the interval $I$ is not a $\sigma'$-interval. Let $k$ be the length of $\sigma$, and for $i \in \{0, 1, \ldots, k-1\}$, let $$D_i = \{a \in D \, : \, S_D^i(a) \in P_\sigma\}.$$ Then $D = D_0 \cup \ldots \cup D_{k-1}$, so it suffices to show that each set $D_i$ is pseudo-arithmetic. 
 
 \begin{claim}
 if $i \in \{0, \ldots, k-1\},$ then $S_{D_i}(a) = S^k_D(a).$
 \end{claim}
 
 \begin{proof}
 Suppose that $0 \leq i < k$, $a \in D$, $S^i_D(a) \in P_\sigma$, and $j \in \omega$ is minimal such that $j > i$ and $S^j_D(a) \in P_\sigma$. Since $|\sigma| = k$, it follows that $S^{i+k}_D(a) \in P_\sigma$, so $j \leq i + k$. If $j - i < k$, then the infinite sequence $\gamma_D(a), \gamma_D(S_D(a)), \ldots, \gamma_D(S^\ell_D(a)), \ldots$ would be $(j-i)$-periodic, and thus $D \cap I$ would be $\sigma'$-periodic where $\sigma'$ is a subsequence of $\sigma$ of length $j-i$, contradicting the minimality of $k$. Therefore $j = i + k$, and the Claim follows.
 \end{proof}
 
 By the Claim, if $\sigma = (d_1, \ldots, d_k)$ and $\eta = \sum_{i=1}^k d_i$, then $D_i$ is $\eta$-pseudo-arithmetic, as we wanted.
\end{proof}

\subsection{Deducing Theorem~\ref{union-arith} from Proposition~\ref{weak-union-arith}}

In this section we show how Theorem~\ref{union-arith} follows from Proposition~\ref{weak-union-arith}.    This will be achieved by showing that after partitioning and performing simple definable transformations on the discrete set $D$ we may reduce to the situation of discrete sets of the  type occurring in Proposition~\ref{weak-union-arith}.  

Throughout the subsection we fix $D\subseteq R$ a definable discrete set. Note that Assumption~\ref{3.1} of the previous subsection is no longer in effect, so $D$ is not necessarily bounded below.  But as noted in Proposition~\ref{sat_assumption}  in this subsection we still  maintain:
\begin{ass}  $\mc{R}$ is $\omega$-saturated.
\end{ass}

The following two Lemmas are immediate.  Recall that if $D$ is discrete and $a \in R$ then $D-a=\{b-a : b \in D\}$ and $-D=\{-b: b \in D\}$.

\begin{lem}\label{unif-move} \begin{enumerate} \item If $D$ is a definable discrete set and $a \in R$ then $\gamma_D(b)=\gamma_{D-a}(b-a)$ for any $b \in D$ and hence $D'=(D-a)'$.   In particular if  $D$ is narrow so is $D-a$.  Also if $D$ is $\eta$-pseudo-arithmetic then so is $D-a$.

\item If $D$ is definable and discrete, $a \in D$, and $S_D(a)=b$ then $S_{-D}(-b)=-a$ and $\gamma_{-D}(-b)=\gamma_D(a)$.  Hence $D'=(-D)'$.   In particular if $D$ is narrow so is $-D$ and if $D$ is $\eta$-pseudo-arithmetic then so is $-D$.
\end{enumerate}
\end{lem}

\begin{lem}\label{moves}
\begin{enumerate}
\item If $D$ is the finite union of points and definable discrete sets each of which is piecewise pseudo-arithmetic then $D$ is piecewise pseudo-arithmetic.

\item If $D$ is piecewise pseudo-arithmetic and $a \in R$ then $D-a$ is also piecewise pseudo-arithmetic;

\item $D$ is piecewise pseudo-arithmetic then so is $-D$.

\end{enumerate}

\end{lem}

\begin{prop}\label{reduce-union-arith}  To show that $D$ is piecewise pseudo-arithmetic it suffices to show that any definable discrete $E$ which is bounded below is a finite union of points and definable discrete sets $E_i$ which are narrow.
\end{prop}

\begin{proof}  Let $D$ be an arbitrary definable discrete set.   Let $D_0=D \cap R^{\geq 0}$ and let $D_1=D\setminus D_0$.  By Lemma~\ref{moves}(1) it suffices to show that $D_0$ and $D_1$ are piecewise pseudo-arithmetic. 
 By Lemma~\ref{moves}(3) to show that $D_1$ is piecewise pseudo-arithmetic if suffices to show that $-D_1$ is piecewise pseudo-arithmetic.  Thus we may assume that $D$ is bounded below.   By assumption $D$ is a finite union of points and  definable discrete sets $E_i$ for $1 \leq i \leq r$ each of which is narrow.   By Lemma~\ref{moves}(1) it suffices to show that  $E_i$ is piecewise pseudo-arithmetic for each $1 \leq i \leq r$. Fix $1 \leq i \leq r$.   But Proposition~\ref{weak-union-arith} applies to $E_i$ and so $E_i$ is piecewise pseudo-arithmetic.  Hence $D$ is piecewise pseudo-arithmetic.
\end{proof}

Thus for the rest of this subsection we work with:

\begin{ass} $D$ is bounded below.
\end{ass}

  We aim to show that $D$ is a finite union of definable subsets $D_i$ which are narrow.  In order to establish this we first show that $D$ can be decomposed into finitely many definable subsets which can be considered highly ``uniform''.   To this end we need some definitions (recall the definition of $D$-convex--see Definition~\ref{dconvex}).

\begin{definition}
 A \emph{finite convex partition of $D$} is a finite partition $D = D_0 \cup \ldots \cup D_m$ such that each $D_i$ is $D$-convex.  We say that the partition is {\em definable} if each $D_i$ is.
\end{definition}
  \begin{definition}  A definable discrete set $D$ is {\em uniformized} if either $D$ is a singleton or $D$ is infinite  and there is $N(D) \in \omega$ so that if $C \subseteq D$ is $D$-convex and of size at least $N(D)$ then for all $\delta \in D'$ there is $a \in C$ with $\gamma_D(a)=\delta$.
\end{definition}

Our first goal is to prove:

\begin{prop}\label{uniformized}  If $D$  is a definable discrete set which is bounded below, then $D$ has a definable finite convex partition into subsets $D_1, \dots, D_n$  so that each $D_i$ is uniformized.
\end{prop}

We begin by establishing the necessary results to show that $D$ is a finite union of uniformized sets.  The following is immediate:

\begin{lem}\label{part-diff} Given $D = D_1 \cup \ldots \cup D_m$ a finite convex partition of $D$.  If $a \in D_i$ for some  $i \in \{1, \ldots, m\}$ and $a$ is not maximal in $D_i$ then $\gamma_{D_i}(a)=\gamma_D(a)$.  In particular we have that $D'_i \subseteq D'$.   
\end{lem}

Notice  that if $E \subseteq D$ is definable and $D$-convex it is potentially the case that  $\gamma_E(\max(E))$ is not defined while $\gamma_D(\max(E))$ is defined and thus it may not be the case that $E'=\{\gamma_D(e) : e \in E\}$.

To begin with,  if $\delta \in D'$ and  $\gamma_D^{-1}(\delta)$ is finite then we can take a definable finite convex partition of $D$ so that if $\gamma_D(a)=\delta$ then $a \notin D_i$ for any infinite $D_i$ in the partition.   Thus by Lemma~\ref{part-diff} $\delta \notin D_i'$ for any infinite $D_i$ in the partition.  Thus, as $D'$ is finite, we may make the following assumption:

\begin{ass}
\label{infinite_delta}
For each $\delta \in D'$, there are infinitely many elements $a \in D$ with $\gamma_D(a) = \delta$.
\end{ass}

We need a definition.

\begin{definition} For $X \subseteq D$ and $\detla_1 \dots \delta_l \in D'$ we say that $X$ is an {\em anti-$\{\delta_1, \dots, \delta_l\}$ set} if for no $a \in X$ is $\gamma_D(a)=\detla_i$ for $1 \leq i \leq l$. If $l = 1$, we omit the braces and call this an \emph{anti-$\delta_1$ set}. A subset $E \subseteq D$ is called an \emph{anti-$\{\delta_1, \ldots, \delta_l\}$-component} if $E$ is a $D$-convex,   anti-$\{\delta_1, \dots, \delta_l\}$ set, and
if $F \subseteq D$ is $D$-convex with $E \subset F$ then $F$ is not an anti-$\{\delta_1, \ldots, \delta_l\}$ set.
\end{definition}

To establish Proposition \ref{uniformized} we first show that for any $\{\delta_1, \dots, \delta_l\} \subseteq D'$ there is a finite bound on the size of any {\em finite} anti-$\{\delta_1, \dots, \delta_l\}$ component (Lemma \ref{antigapsize}).  In turn we will use this to establish that after potentially partitioning $D$ into definable $D$-convex subsets we can reduce to the case that for any $\delta \in D'$ there are no infinite anti-$\delta$ components (Lemma \ref{antipart}).  Once these two facts are established Proposition \ref{uniformized} is immediate.

 For the finite case we need a simple fact.  

\begin{lem}  Let $\mc{R}$ be an OAG and let $\delta_1, \dots, \delta_l \in R^{>0}$.  If $Z \subseteq \N^l$ is infinite then 
\[S=\left\{ m_1 \detla _1 + \dots + m_l\delta_l : (m_1, \dots, m_l) \in Z\right\}\] is infinite.
\end{lem}

\begin{proof}  We induct on $l$.  If $l=1$ the result is immediate.   Thus suppose we have $\delta_1 < \dots <\delta_{l+1}$.  Let $\pi_{l+1} : \N^{l+1} \to \N$ be the projection onto the last coordinate and first  suppose that $\pi_{l+1}[Z]$ is finite.   
Let $\pi_{(1 \dots l)}: \N^{l+1} \to \N^l$ be the projection onto the first $l$ coordinates.  There is $m^* \in \pi_{l+1}[Z]$ whose pre-image is infinite.  Let $Z_0:=\pi_{(1\dots l)}[\pi_{l+1}^{-1}(m^*)]$.   Then by induction 
\[\{ m_1 \detla _1 + \dots + m_l\delta_l+m^*\delta_{l+1} : (m_1 \dots m_l) \in Z_0\} \subseteq S\] is infinite.  

Therefore we may assume that $\pi_{l+1}[Z]$ is infinite.  As $\detla_{l+1}$ is maximal among the $\detla_i$'s any sum of the form 
$m_1\detla_1 + \cdots + m_{l+1}\detla_{l+1}$ with $m_{l+1} \not= 0$ is Archimedean equivalent to $\detla_{l+1}$ (recall all the $\detla_i$'s are positive and $Z \subseteq \N^{l+1}$).  As $\pi_{l+1}[Z]$ is infinite, the set $S$ must be cofinal in the Archimedean class $[\detla_{l+1}]$ of $\delta_{l+1}$ and thus infinite.
\end{proof}

We note in passing that by a theorem of Levi, an Abelian group is orderable just in case it is torsion-free \cite{Levi}, and thus the conclusion of the previous lemma holds for any torsion-free Abelian group.

\begin{lem}\label{antigapsize}  For any $\detla_1 \dots \delta_l \in D'$ there is $N \in \omega$ so that if $X \subseteq D$ is a finite anti-$\{\delta_1\dots \delta_l\}$-component, then $|X|< N$.
\end{lem}

\begin{proof}
Suppose the lemma is false and take $\detla_1, \ldots, \delta_l$ so that $D$ has arbitrarily large finite anti-$\{\delta_1, \ldots, \delta_l\}$-components. Also let $D' = \{\delta_1, \ldots, \delta_n\}$ for $n \geq l$ (recalling that $D'$ must be finite). Let 
\[D_0=\{x \in D: \gamma_D(x) \in \{\delta_1, \ldots, \delta_l\} \}.\] 

Let $E$ be a finite anti-$\{\delta_{1} \dots \delta_{l}\}$-component of $D$ and further assume that $\min{(E)}>\min(D)$.  Let $M=|E|$ and let $a$ be the minimal element of $E$.  Now note that $\gamma_{D_0}(S^{-1}_D(a))=m_1\delta_1 + \cdots m_n\delta_n$ where $m_1 + \cdots +m_n=M+1$. 
As the size of finite  anti-$\{\delta_{1} \dots \delta_{l}\}$-components of $D$  is unbounded there must be an infinite $Z \subseteq \N^n$ so that if $(m_1 \dots m_n) \in Z$ then $m_1 \delta_1 + \dots + m_n\delta_n \in D_0'$.  By the previous lemma $D_0'$ is infinite, which is impossible since $D_0$ is a definable discrete set.

\end{proof}

Thus we have a uniform bound on finite anti-$\{\delta_1, \dots, \delta_l\}$ components.  To establish the fact on infinite anti-$\detla$ components we need a simple consequence of Theorem~\ref{descending_prime}.

\begin{lem}\label{arch-dif}  If $D_0$ and $D_1$ are infinite discrete definable sets, then it cannot be the case that 
$D_0' \ll D_1'$.
\end{lem}

\begin{proof}  Suppose otherwise.  By Lemma~\ref{unif-move} we may without loss of generality assume that $D_i \subseteq R^{\geq 0}$ and $0 \in D_i$ for $i \in \{0,1\}$.  Set $D_0'=\{\eta_1, \dots, \eta_m\}$.   As $D_0' \ll D_1'$ and $0$ is the minimal element of $D_0$ it follows that if $d$ is among the first $\omega$ many  elements of $D_0$ then $d=l_1 \eta_1 + \cdots +\l_m\eta_m$ for some $l_1 \dots l_m \in \NN$.  In particular $d<D_1'$.  Hence by compactness find $\varepsilon \in R$ with $\varepsilon < D_1'$ so that $[0, \varepsilon) \cap D_0$ is infinite.  Then $(0, \varepsilon) \cap D_0 < D_1'$, violating Theorem \ref{descending_prime}.

\end{proof}

\begin{lem}\label{antipart}  There is a definable finite convex partition  $D_0 \cup \ldots \cup D_m$ of $D$ so  that for every $D_i$ and every $\delta \in D_i'$ the set $D_i$ has no infinite anti-$\delta$-component.

\end{lem}

\begin{proof} We proceed by induction on $n=|D'|$. If $n = 1$ then the result is trivial. Assume that $D'=\{\delta_1, \dots, \delta_{n+1}\}$ where the $\delta_i$'s are listed in increasing order, and suppose that we have established the lemma for all definable discrete sets which are bounded below and whose difference sets have size at most $n$.

We will do a sub-induction on $i \in \{1, \ldots, n+1\}$, in reverse order (beginning with $i = n+1$), to show the following:

$(*)_i$ There is a definable finite convex partition of $D = D_0 \cup \ldots \cup D_m$ such that for every $1 \leq \ell \leq m$ and every $j \in \{i, i+1, \ldots, n+1\}$, either $\delta_j \notin D'_l$, or else $D_l$ has no infinite anti-$\delta_j$-component.

Notice that $(*)_1$ is our desired result.

First we will establish the base case, that is, $(*)_{n+1}$. If there are only finitely many infinite anti-$\delta_{n+1}$-components in $D$, then $(*)_{n+1}$ follows quickly: there is a definable finite convex partition $D = D_0 \cup \ldots \cup D_m$ such that for each $i,$ either (i) $D_i$ is an infinite anti-$\delta_{n+1}$-component, or else (ii) $D_i$ has no infinite anti-$\delta_{n+1}$-components.

Thus we may assume that there are infinitely many infinite anti-$\delta_{n+1}$-components.

Let \[E=\{a \in D : a \text{ lies in an infinite anti-}\delta_{n+1} \text{-component}\}\] which is definable by Lemma \ref{antigapsize}.  Let 
\[E_0=\{a \in E : \gamma_D(S^{-1}_D(a))=\detla_{n+1}\},\] or in other words, $E_0$ is the set of all initial points of infinite anti-$\delta_{n+1}$-components.   Notice that as we assume that there are infinitely many infinite anti-$\delta_{n+1}$-components $E_0$ is infinite.

 By our assumptions we may find   $k \in \{1, \dots, n+1\}$ minimal so that there are infinitely many infinite anti-$\{\delta_k \dots \delta_{n+1}\}$-components. Notice that $k = 1$ is impossible and thus $k \geq 2$.

Pick $c,d \in D$ such that $F:=[c,d] \cap D$ is infinite and $F' \subseteq \{\delta_1, \dots, \delta_{k-1}\}$.  
Notice that by the minimality of $k$ and the saturation of $\mc{R}$, for infinitely many anti-$\{\delta_k \dots \delta_{n+1}\}$-components $C$, there are infinitely many $a \in C$ with $\gamma_D(a)=\delta_{k-1}$.  

Thus there are infinitely many $a \in E_0$ such that $S_{E_0}(a) - a \gg \delta_{k-1}$, so by compactness we can pick $\zeta \gg \detla_{k-1}$ such that there are infinitely many $a \in E_0$ satisfying $S_{E_0}(a)-a>\zeta$.  Let
\[E_1=\{a \in E_0 : S_{E_0}(a)-a>\zeta\}.\]  $F$ and $E_1$ are infinite discrete sets so that $F' \ll E_1'$, contradicting  Lemma \ref{arch-dif}. This concludes the proof of $(*)_{n+1}$.

Now assume that $(*)_i$ holds for $i > 1$, and for the inductive step we need to show $(*)_{i-1}$.  After partitioning as in $(*)_i$, we may assume that in $D$ itself there are no infinite anti-$\delta_j$-components for any $j \geq i$. As in the proof of $(*)_{n+1}$, the case where  $D$ has only finitely many infinite anti-$\detla_{i-1}$-components is straightforward.  Thus we assume  that $D$ has infinitely many infinite anti-$\delta_{i-1}$-components. By further partitioning $D$ as necessary and applying induction on $|D'|$ we may assume that in any infinite anti-$\delta_{i-1}$-component of $D$, there are infinitely many points $a$ with  $\gamma_D(a)=\detla_{n+1}$.  Now set
\[E=\{ a \in D : a \text{ lies in an infinite anti-}\delta_{i-1} \text{-component}\}\] which, again, is definable by Lemma \ref{antigapsize}.  Set
\[E_0=\{a \in E : S^{-1}_D(a) \notin E\}.\]  Notice that $E_0$ is infinite and if $a \in E_0$ is not maximal then 
$S_{E_0}(a)-a\gg \delta_{n+1}$.  Now $D$ and $E_0$ are infinite discrete sets such that 
$D' \ll E_0'$, contradicting Lemma \ref{arch-dif}.  This completes the sub-induction and hence establishes the Lemma.
\end{proof}

\begin{proof}(Of Proposition \ref{uniformized})
Partition $D$ into $D_1 \dots D_m$ as in Lemma \ref{antipart}.  Now by Lemma \ref{antigapsize} for each $1 \leq j \leq m$ there is $N_j \in \omega$ so that for any $\delta_1 \ldots \delta_l \in D_j'$ if $X$ is an anti-$\{\delta_1 \ldots \delta_l\}$ component then $|X|<N_j$.  It follows that $D_j$ is uniformized.
\end{proof}

Next we need to establish that we may partition a discrete set, $D$, into finitely many definable subsets $D_i$  each of which is narrow.

To achieve this we first show:

\begin{prop}\label{arch-diff}  Suppose that $D$ is a uniformized then  $D$ is a finite union of definable discrete sets $D_1, \dots, D_m$ each of which is narrow.
\end{prop}

\begin{proof}

First we set some notation. If $E$ is a definable discrete set, let $l(E)$ be the (finite) number of distinct Archimedean classes represented in $E'$.  For ease of notation let $l=l(D)$ for our fixed discrete set $D$.  Let $K_1 < \ldots < K_l$ list all the Archimedean classes in $D'$ and let $M(D)$ be the cardinality of the largest $D$-convex subset of $D$ of the form 
$[a,b) \cap D$ with $a,b \in D$  so that for no $c \in [a,b) \cap D$ is $\gamma_D(c) \in K_l$.  As  $D$ is uniformized $M(D)$ is finite.

We prove the Proposition by induction on $l=l(D)$.  If $l=1$ the result is trivial, so suppose that $l>1$ and the conclusion of the Proposition holds  for all discrete sets $E$  which are bounded below and for which   $l(E)<l(D)$. We proceed by a sub-induction on $M(D)$.  Note that $M(D)=0$ is impossible if $l > 1$, so suppose that $M(D)\geq1$ and the proposition is true for all uniformized discrete sets $E$ which are bounded below, for which the Archimedean classes in $E'$ are among $K_1 \dots K_l$,  and are such that $M(E)<M(D)$.

Now define \[D_1=\{a \in D : \gamma_D(a) \in  K_l\} \cup \{\max(D)\}.\] 
(Where $\max(D)$ is added if and only if it exists.)

Since $M(D)$ is finite and $K_l$ is the maximal Archimedean class every element $\delta \in D_1'$ must lie in $K_l$.   Thus $D_1$ is a definable subset of $D$ so that all elements of $D_1'$ lie in the same Archimedean class.   

Now let $\tilde{D}=D \setminus D_1$.  If $a \in \tilde{D}$ and $\gamma_D(S_D(a)) \notin K_l$  then $\gamma_{\tilde{D}}(a)=\gamma_D(a)$ and so $\gamma_{\tilde{D}}(a) \in K_1 \cup \dots \cup K_l$.  Next suppose that $\gamma_D(S_D(a))\in K_l$.  As $D$ is uniformized and $l>1$ there is some $r \in \NN$ so that 
$\gamma_D(S_D^r(a)) \notin K_l$.  In particular as $K_l$ is the maximal Archimedean class it follows that $\gamma_{\tilde{D}}(a) \in K_l$.   Thus the Archimedean classes represented in $\tilde{D}'$ are among $K_1, \dots, K_l$.

We may  apply Proposition~\ref{uniformized} to $\tilde{D}$ and work within one of the $\tilde{D}$-convex pieces, so without loss of generality $\tilde{D}$ is uniformized and the Archimedean classes represented in $\tilde{D}'$ are still among $K_1 \dots K_l$ (which follows from Lemma~\ref{part-diff}). If $l(\tilde{D}) < l(D)$, we are done by induction.  Thus we may assume that all Archimedean classes $K_1 \dots K_l$ are represented in $\tilde{D}'$.

Suppose that $M(\tilde{D}) \geq M(D)$ and let this be witnessed by $[a,b) \cap \tilde{D}$ with $a,b \in \tilde{D}$.   Then $\gamma_D(b) \notin K_l$ and $b$ is not maximal in $D$.  Next suppose that $c\in [a,b) \cap D$ and $\gamma_D(c) \in K_l$.   As $D$ is uniformized and $l>1$ there is some minimal $r\in \NN$ so that $S_D^{-r}(c) \in [a, b) \cap \tilde{D}$.  But then $S_{\tilde{D}}(S^{-r}(c)) \in K_l$ contradicting the definition of $M(\tilde{D})$.  Hence for every $c \in [a,b) \cap  D$ we have $\gamma_D(c) \notin K_l$.   But then $[a,S_D(b))$ is of size at least $M(D)+1$ and $\gamma_D(c) \notin K_l$ for all $c \in [a,S_D(b))$, a contradiction.  Thus $M(\tilde{D})<M(D)$ and we are done by induction.
\end{proof}

The final remaining element needed in  order to establish Theorem~\ref{union-arith} is:

\begin{prop}\label{good-part}  If $D$ is a definable discrete set which is bounded below then $D$ is a finite union of definable, narrow, discrete sets.
\end{prop}

\begin{proof} First apply Proposition~\ref{uniformized} to $D$.  Thus we can write $D=\bigcup_{i=1}^nD_i$ where the $D_i$ form a definable convex partition of $D$ and each $D_i$ is uniformized.   Fix $i$, as $D_i$ is bounded below  apply Proposition~\ref{arch-diff} to write $D_i$ as a finite union of definable sets which are narrow.   But then $D$ is a finite union of definable, narrow, discrete sets.

\end{proof}

\begin{proof}(Of Theorem~\ref{union-arith})  The theorem follows immediately from 
  Proposition~\ref{reduce-union-arith} and  Proposition~\ref{good-part}.
\end{proof}

\subsection{Proof Theorem \ref{def_in_G} and an example}\label{refined}

In this final subsection we prove Theorem \ref{def_in_G} and finish the subsection highlighting this theorem's limits.    As in the previous section we work in a structure $\mc{R}=\langle R, +, <, \dots\rangle$ but we no longer assume any saturation for $\mc{R}$.

We need some basic results about the structure of pseudo-arithmetic sets.

\begin{lem}\label{trunc} Suppose that $E_0$ and $E_1$ are both definable $\eta$-pseudo-arithmetic sets and that $0$ is the least element of both $E_0$ and $E_1$. Then either $E_0$ is an initial segment of $E_1$ or $E_1$ is an initial segment of $E_0$.
\end{lem}

\begin{proof}  We first show that either $E_0 \subseteq E_1$ or $E_1 \subseteq E_0$.   Suppose that $E_0 \not\subseteq E_1$.  Let $\alpha$ be the least element of $E_0 \setminus E_1$.
Note that $\alpha$ must have a predecessor, $\alpha-\eta$, in $E_0$ (and hence $\alpha-\eta$ is  also in $E_1$).
If $S_{E_1}(\alpha-\eta)$ is defined then $S_{E_1}(\alpha-\eta)=\alpha-\eta+\eta=\alpha$ and $\alpha \in E_1$.  Hence $\alpha-\eta$ must be the maximal element of $E_1$.  If $E_1$ is not a subset of $E_0$ pick $\gamma$ maximal in $E_1$ not in $E_0$.  Then $\gamma<\alpha-\eta$ and so $\gamma+\eta \in E_1$ and notice that also $\gamma+\eta \in E_0$.  But then immediately $\gamma \in E_0$, a contradiction.  Hence $E_1 \subseteq E_0$.

Suppose that  $E_1 \subseteq E_0$, we show that $E_1$ is an initial segment of $E_0$. Let  $\alpha \in E_0$ and $\alpha \leq \max{E_1}$.  If $\alpha=0$ then $\alpha \in E_1$ otherwise pick $\beta \in E_1$ with $\beta$ maximal so that $\beta < \alpha$.  Notice that $\beta$ is not maximal in $E_1$.   As both $E_0$ and $E_1$ are $\eta$-pseudo-arithmetic   $S_{E_1}(\beta)\leq\alpha$  but $S_{E_1}(\beta)<\alpha$ violates the maximality of $\beta$ and hence $S_{E_1}(\beta)=\alpha$ and so $\alpha \in E_1$.  So $E_1$ is an initial segment of $E_0$.

\end{proof}

As an immediate consequence of Lemma \ref{arch-dif} we have:

\begin{lem} \label{arith-arch} If $E_1$ is a definable $\eta_1$-pseudo-arithmetic set and $E_2$ is a definable $\eta_2$-pseudo-arithmetic set then $\eta_1$ and $\eta_2$ are Archimedean equivalent.
\end{lem}

\begin{lem}\label{rat} Suppose that $E_1$ is a definable $\eta_1$-pseudo-arithmetic set and $E_2$ is a definable $\eta_2$-pseudo-arithmetic set then 
there is $q \in \Q$ so that $\eta_2=q\eta_1$.
\end{lem}

\begin{proof}  For each $E_i$ at least one of $E_i \cap [0, \infty)$ or $E_i \cap (-\infty, 0]$ is $\eta_i$-pseudo-arithmetic, hence we may assume that for each $i$ either $E_i \subseteq R^{\geq 0}$ or $E_i \subseteq R^{\leq 0}$.  By Lemma~\ref{unif-move} we may without loss of generality assume that $0 \in E_i$ and $E_i \subseteq R^{\geq 0}$ for $i \in \{1,2\}$.  Also by Lemma \ref{arith-arch} $\eta_1$ and $\eta_2$ are Archimedean equivalent.  Notice that the first $\omega$ many elements of $E_i$ are of the 
form $n \eta_i$ for $n \in \NN$.  If $n\eta_1=m\eta_2$ for some $m,n \in \NN^{>0}$ then we are done.  Let $F=E_1 \cup E_2$, which is also discrete.   As $\eta_1$ and $\eta_2$ are Archimedean equivalent it must be the case that for infinitely many $n_i \in \NN$ that the successor of $m_i \eta_1$ in $F$ is $n_i \eta_2$ for some $m_i, n_i \in \NN$.  Also as $F'$ is finite we can find $\gamma \in R$ and infinitely many pairs $(m_i, n_i) \in \NN^2$ so that 
$S_F(m_i\eta_1)=m_i\eta_1+\gamma=n_i\eta_2$.  Thus working with $(m_0,n_0)$ and $(m_1,n_1)$ we have that $m_0\eta_1+\gamma=n_0\eta_2$ and $m_1\eta_1+\gamma=n_1\eta_2$.  Eliminate $\gamma$ and solve for $\eta_2$ to get:  
$\eta_2=\frac{(m_0-m_1)}{(n_0-n_1)}\eta_1$ as desired.

\end{proof}

\begin{thm} \label{def-arith} Let $\mc{R}=\langle R; +, <, \dots\rangle$ be a definably complete  expansion of a divisible ordered Abelian group of burden $2$.  Let $D$ be a definable infinite discrete set.  There is a definable pseudo-arithmetic set $E \subseteq R^{\geq 0}$ with minimal element $0$ so that $D$ is definable in 
$\langle R; +, <, E\rangle$.
\end{thm}

\begin{proof}  By Theorem \ref{union-arith} $D$ is a finite union of definable pseudo-arithmetic sets, $D_1 \dots D_n$.   Cleary, $D$ is definable in any structure in which $D_1, \dots, D_n$ are definable.  As each $D_i$ is definable in any structure where $D_i \cap (-\infty, 0]$ and $D_i \cap [0, \infty)$ are both definable we may replace each $D_i$ by $D_i \cap (-\infty, 0]$ and $D_i \cap [0, \infty)$ and thus we find pseudo-arithmetic sets $E_1 \dots E_l$ so that for each $1 \leq i \leq l$ either $E_i \subseteq R^{\geq 0}$ or $E_i \subseteq R^{\leq 0}$ and $D$ is definable in $\langle R; +, <, E_1, \dots E_l\rangle$.   Further each $E_i$ is interdefinable with its image under a non-zero $\Q$-linear map and hence after translating and possibly multiplying by $-1$ we may replace $E_i$ by pseudo-arithmetic sets $F_i$ so that $\min(F_i)=0$ for each $i$ and so that $D$ is definable in $\langle R; +, <, F_1, \dots F_l\rangle$.

As $F_i$ is interdefinable with $qF_i$ for any  $q \in \Q^{\not=0}$ we may further assume by   Lemma  \ref{rat} that all of the $F_i$ are $\eta$-psuedo-arithmetic for one fixed $\eta$.   By Lemma \ref{trunc} for each $i$ and $j$ either $F_j$ is an initial segment of $F_i$ or vice versa.  Thus for some $F$ among the $F_i$, each $F_i$ is an initial segment of $F$.  Hence each $F_i$ and thus also $D$ is definable in $\langle R; +,<,F\rangle$.
\end{proof}

 %

Notice that it follows from  Theorem \ref{def-arith} and Lemma \ref{rat} that we can fix a single $\eta \in R$ so that if $D$ is any discrete set definable in $\mc{R}$ then there is an $\eta$-pseudo-arithmetic set $E$ so that $D$ is definable in $\langle R; +, <, E \rangle$.

Now we work towards showing that all infinite discrete sets $D$  definable in $\mc{R}$ are definable in a model of $\Th(\langle \R; +, <, \ZZ\rangle)$.  We recall some basic facts about  $T_{\ZZ}=\Th(\langle \R; +, <, \ZZ\rangle)$. 
\begin{fact}
\label{R0_facts}
\begin{enumerate}
\item \cite{ivp} The complete theory of the structure \[\langle \R; +, <, 0, 1, \lfloor \,\, \rfloor, \lambda \, : \, \lambda \in \Q \rangle\] has quantifier elimination, where $0$ and $1$ are constants, $\lfloor x \rfloor$ is the unary ``floor'' function giving the greatest integer less than or equal to $x$, and $\lambda$ denotes the unary function sending $x$ to $\lambda \cdot x$.   Notice that this structure has the same definable sets  as $\langle \R; +, <, \Z\rangle$ and is definably complete.   Furthermore in this language the structure is universally axiomatizable and hence all definable functions are given piecewise by terms.
\item \cite[Section 3.1]{DG} The theory $T_{\Z}$ has dp-rank $2$.
\end{enumerate}
\end{fact}

We work towards identifying a discrete subgroup of $\mc{R}$ which will act as our copy of ``$\Z$''.

\begin{lem}\label{groupish}  Suppose that $E \subseteq R$ is a definable $\eta$-pseudo-arithmetic set with smallest element $0$. Let $\mc{H}$ be the convex hull of $E$.
\begin{enumerate}
\item If $a,b \in E$ and $a+b \in \mc{H}$, then $a+b \in E$;
\item If $a, b \in E$ and $a\leq b$, then $b-a \in E$.
\end{enumerate}

\end{lem}

\begin{proof}  (1):  Suppose this fails.  Let $a \in E$ be least so that there is $b \in E$ with $a+b \in \mc{H}$ but $a+b \notin E$.  Notice that $0<a$, and let $a^*$ be the predecessor of $a$ in $E$.  Then $a^*+b \in \mc{H}$ so $a^*+b \in E$.  Also $a^*+b$ can not be maximal in $E$ as $a^*+b<a+b$ and $a+b$ is in the covex hull of $E$.   Thus 
$a^*+b+\eta \in E$.  But $a^*+b+\eta=a+b$, a contradiction.

(2)  Suppose this fails.  Let $a \in E$ be minimal so that $b-a \notin E$ for some $b \in E$ with $b\geq a$.  Notice that $a>0$, and if $a^* = S_E^{-1}(a)$, we must have that $b-a^* \in E$. Also $b-a^*>0$ and thus 
$b-a^*-\eta \in E$, but $b-a^*-\eta = b-a$, a contradiction.

\end{proof}

\begin{definition} Let $\eta \in R^{>0}$.  A subgroup $G \subseteq R$ is {\em $\eta$-integer-like} if $G$ is a discrete subgroup of $R$, $\eta$ is the smallest positive element of $G$, and for any $a \in R$ in the convex hull of $G$  there is    $b \in G$ with $b \leq a \leq b+\eta$.  $G$ is {\em integer-like} if it is $\eta$-integer-like for some $\eta$.
\end{definition}

If a definable, discrete, $\eta$-pseudo-arithmetic set $E \subseteq R^{\geq 0}$ is unbounded and  has minimal element $0$ then by Lemma \ref{groupish} $E \cup -E$ is an $\eta$-integer-like  subgroup of $R$.  Furthermore if $F$ is any other definable $\eta$-pseudo-arithmetic set with smallest element $0$ then $F$ is an initial segment of $E$ by Lemma \ref{trunc}.  We aim to show that in general there is an unbounded $\eta$-integer-like subgroup $G \subseteq R$ so that for all $E$, an $\eta$-pseudo-arithmetic set with smallest element $0$, $E$ is an initial segment of $G^{\geq 0}$.  Of course $G$ will typically not be definable.

\begin{lem}\label{group}   Suppose that $\eta \in R$ and there is a definable $\eta$-pseudo-arithmetic set.  Then there is an unbounded $\eta$-integer-like subgroup, $G$, so that for all $\eta$-pseudo-arithmetic definable sets $E$ with smallest element $0$,  $E$ an initial segment of $G^{\geq 0}$.
\end{lem}

\begin{proof} 

  First suppose there is a definable unbounded $\eta$-pseudo-arithmetic set $E$.  By Lemma~\ref{unif-move}(1) we may translate $E$ to   assume $0 \in E$.  Then by Lemma~\ref{groupish} $(E \cap [0, \infty)) \cup -(E \cap [0, \infty))$ is the desired group.  Also notice that by Lemma~\ref{trunc} $(E \cap [0, \infty) \cup -(E \cap [0,\infty))$ must be the unique subgroup with the desired properties.  Hence we assume that any definable $\eta$-pseudo-arithmetic set is bounded.  Again by Lemma~\ref{unif-move} if there is a definable $\eta$-pseudo-arithmetic set then there is one with least element $0$.
  
     Let 
 \[F_0=\bigcup\{E : E \text{ definable, } \eta\text{-pseudo-arithmetic, and}\min(E)=0\}.\]  Note that by Lemma \ref{trunc} $F_0=\bigcup_{i \in \alpha} E_i$ where $\alpha$ is an ordinal,  each $E_i$ is definable and $\eta$-pseudo-arithmetic with minimal element $0$,  and if $i<j \in \alpha$ then $E_i$ is an initial segment of $ E_j$.  Furthermore, again by Lemma~\ref{trunc} if $E$ is any definable $\eta$-pseudo-arithmetic set then $E$ is an initial segment of $F_0$.

We claim that if $x,y \in F_0$ then $x+y \in F_0$.  Pick $i$ so that $x,y \in E_i$ by Lemma \ref{groupish} $x+y \in 
E_i \cup (E_i+c)$ where $c=\max{E_i}$.  Note that $c$ exists as $E_i$ is bounded by assumption and also that $E_i \cup (E_i+c)$ is again $\eta$-pseudo-arithmetic.  But $E_i \cup (E_i+c) \subseteq F_0$.  Similarly if $x,y \in F_0$ and $x<y$ then $y-x \in F_0$.  Thus $F=F_0 \cup -F_0$ is a subgroup of $R$.  We claim that $F$ is $\eta$-integer-like.  To show that $F$ is discrete is suffices to show if $f_1<f_2 \in F_0$ then $f_2-f_1 \geq \eta$.  But this is immediate as $f_1, f_2 \in E_i$ for some $i \in \alpha$ and $E_i$ is an intial segment of $F_0$ by Lemma~\ref{trunc}.  Also clearly $\eta$ is the smallest positive element of $F$ as each $E_i$ is $\eta$-pseudo-arithmetic.  Finally suppose that $u$ is in the convex hull of $F$ and first assume that $u\geq 0$.  Then $u$ is in the convex hull of $E_i$ for some $i \in \alpha$.  Let $b=\max\{e \in E_i : e \leq u\}$, then $b$ is an element of $F$ so that $b \leq u \leq b+\eta$. 
 If $u<0$, find $b \in F$ so that $b\leq -u\leq b+\eta$, then $-(b+\eta) \in F$  is so that $-(b+\eta) \leq u \leq -(b+\eta)+\eta$.

If $F$ is unbounded we are done. Otherwise let $U$ be the convex hull of $F$ in $R$.  As $F$ is a subgroup of $R$ it follows that $U$ is a rational vector subspace of $R$.  Let $H$ be the subspace of $R$ so that  $R=U\oplus H$ as rational vector spaces.  Then $H$ is a discrete subset of $R$ since if $h_1<h_2 \in H$ and $h_2-h_1<u$ for some non-zero $u \in U$ then, as $U$ is convex, $h_2-h_1=v$ for some non-zero $v \in U$ which is impossible.  Also $H$ is clearly unbounded in $R$.  

Finally let $G=F+H$.  Clearly $G$ is an unbounded subgroup of $R$.  We need to verify that $G$ is $\eta$-integer-like.  Let $f+h \in G$ and note that $(f+\eta)+h \in G$ as $F$ has minimal element $\eta$.  We claim there is no $f'+h' \in G$ so that $f+h<f'+h'<(f+\eta)+h$.  If $h=h'$ then it must be the case that $f'>f$ but we must also have $f'-f<\eta$ which is impossible.  Thus $h \not= h'$, but then $(f'+h')-(f+h) \notin U$ (since otherwise $h'-h \in U$) and so $(f'+h')-(f+h)>\eta$ which is also impossible.  Hence $G$ is a discrete group and $\eta$ must be the smallest element of $G$.  Now suppose that $a \in R$.  Write $a=u+h$ where $u \in U$ and $h \in H$.  As $F$ is $\eta$-integer-like there is $b \in F$ so that $b \leq u \leq b+\eta$ but then $b+h \in G$ is so that $b+h \leq u+h \leq b+\eta+h$.  

  By construction if $E$ is any definable $\eta$-pseudo-arithmetic subset of $R$ with minimal element $0$ then $E $ is an initial segment of $F^{\geq 0}$ and hence of $G^{\geq 0}$.

\end{proof}

Now we may provide a proof of Theorem \ref{def_in_G}.

\begin{proof}  By Lemma~\ref{rat} we find find $\eta \in R$ so that if $F\subseteq R$ is any pseudo-arithmetic set then there is $q \in \Q$ so that $qF$ is $\eta$-pseudo-arithmetic.  Fix such an $\eta$.  For any infinite definable discrete set $D$, by Theorem \ref{def-arith} there is a definable  pseudo-arithmetic set $E(D)\subseteq R^{\geq 0}$ with minimal element $0$ so that 
$D$ is definable in $\langle R; +, <, E(D)\rangle$.  By our choice of $\eta$ we may assume that all such $E(D)$ are  $\eta$-pseudo-arithmetic.    As we assume that there is a definable infinite discrete subset of $R$ then there is a definable $\eta$-pseudo-arithmetic set.  By Lemma \ref{group} there is an unbounded $\eta$-integer-like group $G$  so that $E(D)$ is an initial segment of $G^{\geq 0}$ for every infinite definable discrete set $D$.  Notice that all such $E(D)$ (and hence all $D$) are definable in $\langle R; +, <, G\rangle$.  Finally by the Appendix of \cite{ivp} $\langle R; +, <, G\rangle$ is a model of $T_{\ZZ}$.

For the ``furthermore'' we show that under the additional assumption that $\mc{R}$ is of dp-rank $2$  there is $G \subseteq R$ an  integer-like subgroup so that all $\mc{R}$-definable $X \subseteq R$ are definable in $\langle R; +, <, G \rangle$.   Thus suppose that $\mc{R}$ is of dp-rank $2$.  By \cite[Corollary 2.10]{topological_properties_burden_2} any set definable in $\mc{R}$ is either discrete or has interior. By the first part of the theorem we may fix $G \subseteq R$  an integer-like subgroup so that all $\mc{R}$-definable discrete subsets $D \subset R$ are definable in $\langle R; +, < G \rangle$.  It suffices to show that any open set definable in $\mc{R}$ is definable in $\langle R; +, < , G\rangle$.    Thus let $U$ be an open $\mc{R}$-definable subset of $R$ and without loss of generality assume that $U \not= R$.   As $\mc{R}$ is definably complete, $U$ is a disjoint union of open intervals.  Let $D_0$ be the set of all left endpoints of open intervals in $U$ and let $D_1$ be the set of all right endpoints of intervals in $U$.  Note that $D_0$ and $D_1$ are definable in $\mc{R}$ and as we are assuming that $U \not=R$ at least one of $D_0$ or $D_1$ is non-empty.  The sets $D_0$ and $D_1$ are discrete and definable in $\mc{R}$ and hence also definable in $\langle R; +, <, G \rangle$.

Thus $U$ may be defined in $\langle R; +, <, G \rangle$ as:
\[\{x \in R : \exists d_0 \in D_0 \exists d_1 \in D_1(d_0<x<d_1 \wedge (d_0, d_1) \cap (D_0 \cup D_1)=\emptyset)\}\]
\[\cup \{x \in R : \exists d_1 \in D_1 ( x<d_1 \wedge (-\infty, d_1) \cap (D_0 \cup D_1)=\emptyset)\}\]
\[\cup \{x \in R: \exists d_0 \in D_0(d_0<x \wedge (d_0, \infty) \cap (D_0 \cup D_1)=\emptyset)\}\]
\end{proof}

At this stage we provide a more detailed description of the definable subsets of $R$ by simply developing a precise description of the definable subsets in models of $T_{\Z}$.  We recall the following definition.

\begin{definition}  A structure $\mc{R}=\langle R; <, \dots \rangle$ with $<$ a dense linear ordering without endpoints is called {\em locally o-minimal} if for every definable $X \subseteq R$ and every $x \in X$ there is an open interval $I$ with $x \in I$ so that $X \cap I$ is a finite union of points and intervals.
\end{definition}

\begin{prop} \label{more_precise} If $\mc{M}=\langle M; +, <, Z \rangle \models T_{\Z}$ and $X \subseteq M$ is definable then $X$ is a finite union of sets of the form 
\begin{itemize} \item $\bigcup_{w \in W}\{w+\lambda\}$ where $W$ is definable in the structure $\langle Z; +, < \rangle$ and $\lambda \in [0,1)$ (where $1$ denotes the minimum positive element of $Z$);

or 

\item $\bigcup_{w \in W}(w+\lambda_0, w+\lambda_1)$ where $W$ is definable in the structure 
$\langle Z; +, < \rangle$, $\lambda_0<\lambda_1$, and $\lambda_0,\lambda_1 \in [0,1]$.

\end{itemize}

\end{prop}

\begin{proof} 

By \cite[Lemma 3.3(i)]{DG} and compactness there is $N \in \NN$ so that  if $l \in Z$ then $X \cap (l,l+1)=X_1 \cup \dots \cup X_N$ where each $X_i$ is either  empty, a single point, or an open interval, and the $X_i$ are pairwise disjoint.  We proceed by induction on $N$.

 By \cite[Theorem 19]{KTTT} $T_{\Z}$ is locally o-minimal and as noted in Fact \ref{R0_facts} it is definably complete.  These two facts together imply that  if $X \subseteq M$ is definable then $X$ is a union of a closed discrete set and an open set (see \cite[Lemma 2.3]{tame_lom}).  Furthermore by definable completeness any  open subset is a union of open intervals. %

Thus without loss of generality we may assume that $X$ is either a union of open intervals or closed and discrete. If $Y \subseteq Z$ is definable  then $Y$ is definable in $\langle Z; +, < \rangle$ by \cite[Lemma 3.3(ii)]{DG}.  Thus $X \cap Z$ is a set of the first form in the statement of the Lemma and we may  without loss of generality assume that $X \cap Z = \emptyset$.

  Suppose that $N=1$ and also that $X$ is open.  Thus $X$ consists of at most one open subinterval of $(l,l+1)$ for each $l \in Z$.
 Let $X_l$ be the set of left endpoints of  intervals in $X$ and let $X_r$ be the right endpoints.  Note that $X_l$ and $X_r$ are closed and discrete and thus so is $X_l \cup X_r$.  Let $f: X_l \cup X_r \to [0,1]$ be given by $x \mapsto x-\floor{x}$ if $x \in X_l$ or $x \in X_r \setminus Z$ and $x \mapsto 1$ otherwise.   As noted in Fact \ref{R0_facts} $T_{\Z}$ is of dp-rank $2$ and thus is strongly dependent.  Hence by \cite[Corollary 2.17]{DG}  in models of $T_{\Z}$ forward images of discrete sets under definable functions must be discrete.
  Thus $f[X_l \cup X_r] \subseteq [0,1]$ is discrete and so by \cite[Lemma 3.3(i)]{DG} it is  finite.

  Let $\xi_1< \dots < \xi_s$ list all the elements of $f[X_l \cup X_r]$.    For $1 \leq i < j \leq s$ let
\[W_{ij}=\{z \in Z : (z+\xi_i, z+\xi_j) \subseteq X\}.\]
By \cite[Lemma 3.3(ii)]{DG} $W_{ij}$ is definable in $\langle \Z; +, < \rangle$.  Let
\[X_{ij}=\{x : \exists z (z \in W_{ij} \wedge z+\xi_i < x < z+\xi_j)\}.\]
Thus $X=\bigcup_{1 \leq i,j \leq s}X_{ij}$ and each $X_{ij}$ is of the second form given in the statement of the proposition.

Now suppose $X$ is discrete.  Let $\tilde{X}=\bigcup_{x \in X}(\floor{x},x)$, which is a union of open intervals with at most one interval in each $(l, l+1)$ for $l \in Z$.  Thus the previous case applies and $\tilde{X}$ is a finite union of sets 
$X_i= \bigcup_{w \in W_i}(w+\lambda^i_0, w+\lambda^i_1)$ with $W$ definable in $\langle Z; +, < \rangle$, $\lambda^i_0<\lambda^i_1$, and $\lambda^i_0, \lambda^i_1 \in [0,1)$.  (Note that $\lambda_1^i<1$ as $X \cap Z=\emptyset$.)  For each $i$ let $W'_i=\{w \in W_i : w+\lambda^i_1 \in X\}$, which again is definable in $\langle Z ; +, < \rangle$ by \cite[Lemma 3.3(ii)]{DG}.  Then $X$ is the union of  $\bigcup_{w \in W'_i}\{w + \lambda^i_1\}$ which are of the first form in the statement of the Lemma.
  
  Now suppose that $N>1$.  Let 
  \[X_0=\{x \in X : \neg \exists y,z (\floor{x}<y<z<x \wedge z \notin X \wedge y \in X)\}.\]  Then $X_0$ is a set so that $(l, l+1) \cap X_0$ consists of at most one point or interval 
  for each $l\in Z$ and so the $N=1$ case applies to $X_0$.  Also $(X \setminus X_0) \cap (l, l+1)$ must consist of at most $N-1$ points or intervals for all $l \in Z$ so the induction hypothesis applies to $X \setminus X_0$ and we are done.

\end{proof}
 Note that by the cell decomposition result of Cluckers for sets definable in models of the theory of $\langle \Z; +, < \rangle$ \cite{PrCell}, we can describe the sets $W$ in the conclusion of Proposition~\ref{more_precise} even more explicitly:   Each such $W$ is  a finite union of points and infinite intervals intersected with a coset of $nZ$ for some $n \in \N$.

Recall that by \cite{DG} $\langle \R; +, <, \Z\rangle$ has dp-rank $2$ and is clearly definably complete.  In light of Theorem \ref{def_in_G} it is reasonable to ask if for any   $\mc{R}=\langle R; +, <, \dots\rangle$ which is a  definably complete and  dp-rank $2$ expansion of a divisible ordered Abelian group there is a group $G$ so that $\langle R; +, <, G\rangle \models T_{\Z}$ and $\mc{R}$ is a reduct of $\langle R; +, <, G \rangle$.   We give an example to show this is not the case.

To construct our example of a divisible dp-rank $2$ expansion of an  OAG with definable completeness which is not simply a reduct of  a model of $T_{\ZZ}$, we will use the \emph{simple product} construction as defined in \cite{KTTT}, in particular the \emph{standard simple product} as elaborated by Fujita in \cite{fujita}, whose definition we now recall.

\begin{definition}
\label{standard_simple_product}
(See \cite{fujita}, Definition 2.5) Suppose that $\mc{L}_0$ and $\mc{L}_1$ are two disjoint languages with only relation and constant symbols, and we assume that each language contains at least one constant symbol. Further suppose that $\mc{M}_0$ and $\mc{M}_1$ are structures in the languages $\mc{L}_0$ and $\mc{L}_1$ respectively. The \emph{standard simple product of $\mc{M}_0$ and $\mc{M}_1$} is the structure $\mc{N}$ in a new language $\mc{L}_{sim}$, defined as follows:

\begin{enumerate}
\item The universe $N$ of $\mc{N}$ is the Cartesian product $M_0 \times M_1$.
\item For each constant symbol $c_0 \in \mc{L}_0$ and each constant symbol $c_1 \in \mc{L}_1$, there is a constant symbol $C_{c_0, c_1} \in \mc{L}_{sim}$, which is interpreted in $\mc{N}$ as $$C_{(c_0, c_1)}^{\mc{N}} = (c_0^{\mc{M}_0}, c_1^{\mc{M}_1}).$$
\item For $k \in \{0, 1\}$ and for each $n$-ary relation symbol $R \in \mc{L}_k$, there is an identically named relation symbol $R \in \mc{L}_{sim}$. Letting $\pi_k \, : \, N \rightarrow M_0 \times M_1$ be the projection onto the $k$-th coordinate, we define the interpretation $R^{\mc{N}}$ of $R$ in $\mc{N}$ as $$R^{\mc{N}} = \{ (a_1, \ldots, a_n) \in N^n \, : \, (\pi_k(a_1), \ldots, \pi_k(a_n)) \in R^{\mc{M}_k} \}.$$
\item Finally, $\mc{L}_{sim}$ contains two additional binary relation symbols $\sim_0$ and $\sim_1$ interpreted as equality of the corresponding coordinate projections: for each $k \in \{0,1\}$,

$$\sim_k^{\mc{N}} = \{(a,b) \in N^2 \, : \,  \pi_k(a) = \pi_k(b)\}.$$
\end{enumerate}

\end{definition}

As shown by Fujita \cite{fujita}, the complete theory of the standard simple product $\mc{N}$ can be naturally axiomatized in terms of the complete theories of each $\mc{M}_k$.

	\begin{prop}\footnote{We were alerted by the referee that a similar bound was obtained by Touchard (see Proposition 1.24 of \cite{touch}). However, Touchard's result concerns burden, not dp-rank, so to apply his result, we would need to prove that NIP is preserved under direct products of structures; hence we found it simpler to include our own proof here.}
If $\mc{M}_0$ and $\mc{M}_1$ are structures such that $\textup{dp-rk}(\mc{M}_k) \leq \kappa_k$ for $k \in \{0,1\}$, then if $\mc{N}$ is the standard simple product of $\mc{M}_0$ and $\mc{M}_1$, $$\textup{dp-rk}(\mc{N}) \leq \kappa_0 + \kappa_1.$$
\label{simple_product_bound}
\end{prop}

\begin{proof}
We will use the following characterization of dp-rank found in  \cite[Proposition 4.17]{Guide_NIP} (though the idea goes back to \cite{additivity_dp_rank}):

\begin{facta}
If $\mc{M} \models T$ is sufficiently saturated, $\textup{dp-rk}(T) < \kappa$ if and only if for every $b \in M$ and every family of infinite, mutually indiscernible sequences $(I_t \, : \, t \in X)$, there is some $X' \subseteq X$ such that $|X'| < \kappa$ and the family $(I_t \, : \, t \in X \setminus X')$ is mutually indiscernible over $b$.
\end{facta}
 
We will apply this criterion to the standard simple product $\mc{N}$. By \cite[Theorem 2.8]{fujita}, we may replace our original model $\mc{N}$ by a sufficiently saturated extension and the result will still be a standard simple product of models of $\textup{Th}(\mc{M}_k)$. Now fix any $b = (b^0, b^1) \in N$ and let $(I_t \, : \, t \in X)$ be mutually indiscernible sequences of finite tuples from $N$, say $I_t = \{ \overline{a}_{i, t} \, : \, i \in \omega\}$. Each $\overline{a}_{i,t}$ can be written as $(\overline{a}^0_{i,t}, \overline{a}^1_{i,t})$ where $\overline{a}^k_{i,t}$ is the tuple from $M^k$ obtained by coordinate projection, and let $I^k_t = \{\overline{a}^k_{i,t} \, : \, i \in \omega \}$. By \cite[Lemma 3.4]{fujita}, each family $(I^k_t \, : \, t \in X)$ for $k \in \{0,1\}$ is mutually indiscernible, so by the characterization of dp-rank given in the Fact, there are sets $X^k \subseteq X$ such that $|X^k| < \kappa_k^+$ such that $(I^k_t \, : \, t \in X \setminus X^k)$ are mutually indiscernible over $b^k$. Hence $|X^k| \leq \kappa_k$, and thus if $X' = X^0 \cup X^1$, we have $|X'| \leq \kappa_0 + \kappa_1$, or equivalently, $|X'| < (\kappa_0 + \kappa_1)^+$. By \cite[Lemma 3.4]{fujita}, we have $(I_t \, : \, t \in X \setminus X')$ are mutually indiscernible over $b$, so applying the Fact again, we conclude that $\textup{dp-rk}(T) < (\kappa_0 + \kappa_1)^+$, as desired.
\end{proof}

\begin{prop}
\label{sin_rank_2}
Let $\mc{R}=\langle \R; +, <, \sin{x} \rangle$.  If $T=\Th(\mc{R})$ then $T$ is definably complete and has dp-rank $2$.
\end{prop} 

\begin{proof}
By \cite[Lemma 3.3]{Goodrick_dpmin} the dp-rank of $T$ is at least $2$ as no infinite discrete set can be definable in a dp-minimal expansion of a divisible OAG.

$T$ is definably complete as $\mc{R}$ has universe $\R$.  We establish that $T$ has dp-rank $2$ by showing that it may be realized as a reduct of a simple product which we can verify has dp-rank $2$  via  Proposition \ref{simple_product_bound}.

 Let $\mc{M}_0=\langle \ZZ; +_{Z}, <_{Z}, 0_{Z}\rangle$ where $+_Z, <_Z,$ and $0_Z$ are just the usual addition, order, and $0$ except that we think of $+_Z$ as a relation.  Let $\mc{M}_1=\langle [0, 2\pi); +_s, <_s, \sin_s{x}, E_s, 0_s\rangle$ where $<_s$ and $0_s$ are just the usual order and $0$, $\sin_s{x}$ is the sine function restricted to $[0, \pi]$ thought of as a relation, $E_s$ is a binary relation that holds of $a,b$ if and only if $a+b<2\pi$, and $+_s$ is addition modulo $2\pi$ thought of as a relation.  Note that $\mc{M}_0$  is quasi-o-minimal with definable bounds  (see \cite[Example 2 and Remark after Theorem 3 ]{qom})  thus by  \cite[Corollary 3.3]{DGL}) $\mc{M}_0$ is dp-minimal.  By  \cite[Example 1.6]{rag} $\mc{M}_1$ is o-minimal and thus by \cite[Corollary 3.3]{DGL} it is dp-minimal.   Let $\mc{M}$ be the standard simple product of $\mc{M}_0$ and $\mc{M}_1$.  Recall that $\mc{M}$ has universe $M_0 \times M_1$.

We think of each point $(n, x) \in M_0 \times M_1$ as representing $2\pi n + x \in \R$ and the operations $\sin_2$, $+_2$, as defined below, will correspond to the usual operations of $\sin(\cdot)$ and addition on $\R$. Set  $\sin_2 : M \to M$ to be the $\mc{M}$-definable function given by

 \[\sin_2((n,x))= \begin{cases}  (0, \sin{x}) & \text{ if } 0 \leq x \leq \pi \\ (-1, 2\pi-\sin(x-\pi)) & \text{ if }\pi < x < 2\pi\end{cases}\]  Similarly $+_2$ be the $\mc{M}$-definable binary operation given by:
 
 \[(m, x) +_2 (n,y)=\begin{cases} (m+n, x+ y) & \text{ if } x+y<2\pi \\ (m+n+1, x+ y-2\pi) & \text{ if } x+y \geq 2\pi \end{cases}\] 
 
 Note that the relation $E_s$ is used in defining $+_s$.  Also, the lexicographic ordering $<_{lex}$ on $M_0 \times M_1$, is $\mc{M}$-definable. 
 
 Notice that  $\langle M, +_2, <_{lex}, \sin_2{x}\rangle$ is isomorphic to $\langle \R; +, <, \sin{x} \rangle$ via the map $(n, x) \mapsto 2\pi n+x$.  

By Proposition~\ref{simple_product_bound}, the dp-rank of $\mc{M}$ is at most $2$, and as $\mc{R}$ is isomorphic to a reduct of $\mc{M}$ we conclude that $T$ has has dp-rank $2$.

\end{proof}

Finally we point out that the above example is not a reduct of a model of $T_{\ZZ}$.

\begin{prop}  Let $\mc{R}=\langle \R; +, <, \sin{x}\rangle$.  For no $Z \subseteq \R$ with $\mc{R}_Z=\langle \R; +, <, Z\rangle \models T_{\ZZ}$ is $\mc{R}$ a reduct of $\mc{R}_Z$.
\end{prop}

\begin{proof} Fix $Z \subseteq \R$ so that $\mc{R}_Z \models T_{\ZZ}$.   By the first part of Fact \ref{R0_facts} and \cite[Lemma 3.2]{DG} if $(a,b) \subset \R$ is a bounded interval and $f: (a,b) \to \R$ is definable in $\mc{R}_Z$ then $f$ is piecewise linear.  But clearly $\sin : (0,1) \to \R$ is not piecewise linear hence $\mc{R}$ is not a reduct of $\mc{R}_Z$.
\end{proof}
\bibliography{modelth}

\end{document}